\documentclass[a4paper,twoside,10pt]{article}

\usepackage[a4paper,left=3cm,right=3cm, top=3cm, bottom=3cm]{geometry}
\usepackage[utf8]{inputenc}
\usepackage{mathrsfs}
\usepackage{graphicx}
\usepackage{epstopdf}
\usepackage{xcolor}
\usepackage{amsmath}
\usepackage{amsthm}
\usepackage{amssymb}
\usepackage{bm}
\usepackage{mathabx}
\usepackage{stmaryrd}
\usepackage{float}
\usepackage{bigints}
\usepackage{cite}
\usepackage{color}
\usepackage[abs]{overpic}
\usepackage[font=footnotesize,labelfont=bf]{caption}
\usepackage{cases}
\usepackage{tikz}
\usepackage{algorithm,algpseudocode}
\usepackage{rotating}
\usepackage{blkarray}
\usetikzlibrary{matrix,calc,arrows}
\usepackage[author={Lorenzo}]{pdfcomment}
\usepackage{soul,xcolor}
\usepackage{verbatim}
\usepackage{graphicx}
\usepackage{subcaption}
\usepackage{algorithm}
\usepackage{pifont}
\usepackage{multirow}
\usepackage{caption}
\usepackage{multirow}
\usepackage{bm}
\usetikzlibrary{shapes.geometric}
\usepackage{ifthen}
\usepackage{comment}

\restylefloat{table}
\theoremstyle{plain}
\newtheorem{theorem}{Theorem}[section]

\newtheorem{lemma}[theorem]{Lemma}
\newtheorem{proposition}[theorem]{Proposition}
\theoremstyle{definition}

\theoremstyle{remark}
\newtheorem{remark}{Remark}

\newcommand{\green}[1]{{\color{green}{#1}}} 
\newcommand{\orange}[1]{{\color{orange}{#1}}} 

\newcommand{\Norm}[1]{{\left\|{#1} \right\|}}
\newcommand{\SemiNorm}[1]{{\left|{#1} \right|}}

\newcommand{\jump}[1]{\left[\!\left[#1\right]\!\right]}

\newcommand{\average}[1]{\left\{#1\right\}}

\newcommand{\Normthreebars}[1]{\vvvert{#1}\vvvert_n}
\newcommand{\NormthreebarsDG}[1]{\vvvert{#1}\vvvert_{\text{DG}}}
\newcommand{\Normthreebarshp}[1]{\vvvert{#1}\vvvert_{n}}
\newcommand{\Pbb}{\mathbb{P}}

\newcommand{\card}{\operatorname{card}}
\newcommand{\mesh}{\mathcal{T}_h}

\newcommand{\facets}{\mathcal{F}}

\newcommand{\zOmega}{{0,\Omega}}
\newcommand{\normal}[1]{\mathbf{n}_{#1}}

\newcommand{\dx}[1]{\,\mathrm{d}{#1}}
\newcommand{\mis}[1]{{\left|{#1} \right|}}

\DeclareMathOperator{\Span}{span}
\newcommand{\Rbb}{\mathbb{R}}
\newcommand{\Pbbtaun}{\Pbb_\p(\taun)}
\newcommand{\Pbbpmotaun}{\Pbb_{\p-1}(\taun)}
\newcommand{\antildeparent}[2]{\widetilde{a}_n(#1,#2)}
\newcommand{\atilden}{\widetilde{a}_n}
\newcommand{\anparent}[2]{a_n(#1,#2)}
\newcommand{\an}{a_n}
\newcommand{\aparent}[2]{a(#1,#2)}
\newcommand{\gradh}{\nabla_\h}

\newcommand{\hOmega}{\h_\Omega}
\newcommand{\taun}{\mathcal T_\h}
\newcommand{\taunn}{\mathcal T_n}
\newcommand{\taunnpo}{\mathcal T_{n+1}}
\newcommand{\Fcaln}{\mathcal F_\h}
\newcommand{\Vcaln}{\mathcal V_\h}
\newcommand{\FcalnI}{\Fcaln^I}
\newcommand{\FcalnB}{\Fcaln^B}
\newcommand{\VcalnI}{\Vcaln^I}
\newcommand{\VcalnB}{\Vcaln^B}
\newcommand{\E}{K}
\newcommand{\F}{F}
\newcommand{\Xn}{X_n}
\newcommand{\Bn}{\Phi_n}

\newcommand{\nbf}{\mathbf n}
\newcommand{\nbfE}{\nbf_\E}
\newcommand{\nbfF}{\nbf_\F}
\newcommand{\taubold}{\boldsymbol{\tau}}
\newcommand{\tauboldE}{\taubold_\E}
\newcommand{\hE}{\h_\E}
\newcommand{\hF}{\h_\F}
\newcommand{\FcalE}{\mathcal F^\E}
\newcommand{\VcalE}{\mathcal V^\E}
\newcommand{\nablah}{\nabla_\h}
\newcommand{\qp}{q_\p}

\newcommand{\qpmoF}{q_{\p-1}^\F}
\newcommand{\Ihat}{\widehat I}
\newcommand{\Vn}{V_n}
\newcommand{\Vng}{V_n^g}
\newcommand{\Vnzero}{V_n^0}

\newcommand{\Vnimtilde}{\widetilde{V}_n^{\rm im}}
\newcommand{\Vnextilde}{\widetilde{V}_n^{\rm ex}}
\newcommand{\Qn}{\widetilde Q_n}
\newcommand{\PbbB}{\Pbb_\p^B(\mesh) }
\newcommand{\Vnex}{\Vn^{\rm ex}}
\newcommand{\Vnim}{\Vn^{\rm im}}
\newcommand{\vn}{v_n}

\newcommand{\wn}{w_n}

\newcommand{\un}{u_n}
\newcommand{\numedi}{N^I_\F}
\newcommand{\numedt}{N^T_\F}
\newcommand{\numel}{N_\E}
\newcommand{\numver}{N_V}

\newcommand{\nuhat}{\widehat \nu}
\newcommand{\Fhati}{\widehat \F_i}
\newcommand{\phijF}{\varphi_j^\F}
\newcommand{\phihatjFhati}{\widehat\varphi_j^{\Fhati}}
\newcommand{\phihatjVhati}[1]{{\widehat\varphi^{\nuhat_#1}}}

\newcommand{\phiellE}{\varphi_\ell^\E}

\newcommand{\shapereg}{\rho}
\newcommand{\trace}{\gamma_\Omega}
\newcommand{\traceEo}{\gamma_{\E_1}}
\newcommand{\traceEt}{\gamma_{\E_2}}
\newcommand{\traceE}{\gamma_{\E}}

\newcommand{\hoOm}{_{1,\h; \Omega}}
\newcommand{\PiltwoF}[1]{\Pi^{0,\F}_{#1}}
\newcommand{\PiLpmoF}{\Pi_{\Lpmo^\F}}

\newcommand{\omegaF}{\omega_\F}
\newcommand{\lambdaEi}{\lambda_{\E,i}}

\newcommand{\lambdaEF}{\lambda_{\E,\F}}
\newcommand{\lambdahatFhati}[1]{\widehat\lambda_{\widehat\F_{#1}}}
\newcommand{\bpF}{b_\p^\F}

\newcommand{\bpmoF}{b_{\p-1}^\F}
\newcommand{\Lp}{L_\p}

\newcommand{\Lj}{L_j}
\newcommand{\LjF}{\Lj^\F}
\newcommand{\Lpmo}{L_{\p-1}}
\newcommand{\Hone}{H^1(\Omega)}
\newcommand{\Honezd}{[H^1_0(\Omega)]^2}
\newcommand{\Ltwod}{[L^2(\Omega)]^2}
\newcommand{\Ltwoz}{L^2_0(\Omega)}

\newcommand{\qtw}{q_2}
\newcommand{\bpE}{b_\p^\E}
\newcommand{\Hsnctaunom}[2]{H^{#1,nc}_{#2} (\taun,\Omega)}
\newcommand{\Hsncgtaunom}[2]{H^{#1,nc}_{#2,g} (\taun,\Omega)}
\newcommand{\Htildesnctaunom}[2]{\widetilde H^{#1,nc}_{#2} (\taun,\Omega)}
\newcommand{\Htildesncgtaunom}[2]{\widetilde H^{#1,nc}_{#2,g} (\taun,\Omega)}
\newcommand{\weakder}{D^{\boldalpha}}
\newcommand{\FcalEhat}{\mathcal{F}^{\widehat \E}}
\newcommand{\Lebp}{L^\p(\Omega)}
\newcommand{\boldalpha}{\boldsymbol \alpha}
\newcommand{\Ehat}{\widehat \E}
\newcommand{\Fhat}{\widehat \F}

\newcommand{\Vtilden}{\widetilde V_n}
\newcommand{\Vbftilden}{\widetilde {\bf V}_n}
\newcommand{\Vtildeng}{\widetilde V_n^g}
\newcommand{\Vtildenexg}{\widetilde V_n^{ex,g}}
\newcommand{\Vtildenzero}{\widetilde V_n^0}
\newcommand{\vI}{v_I}

\newcommand{\boldbeta}{\boldsymbol \beta}
\newcommand{\mboldbetaE}{m_{\boldbeta}^\E}
\newcommand{\qpEo}{q_{\p}^{\E_1}}

\newcommand{\qpEtw}{q_{\p}^{\E_2}}

\newcommand{\qpEi}{q_{\p}^{\E_i}}

\newcommand{\qzEo}{q_{0}^{\E_1}}
\newcommand{\qzEtw}{q_{0}^{\E_2}}
\newcommand{\qzEi}{q_{0}^{\E_i}}
\newcommand{\utilden}{\widetilde u_n}
\newcommand{\vtilden}{\widetilde v_n}
\newcommand{\qtilden}{\widetilde q_n}
\newcommand{\stilden}{\widetilde s_n}
\newcommand{\Qtilden}{\widetilde Q_n}
\newcommand{\Bcaltilden}{\widetilde{\mathcal B}_n}
\newcommand{\btilden}{\widetilde{b}_{n}}
\newcommand{\ctilden}{\widetilde{c}_{n}}
\newcommand{\bbftilden}{\widetilde{\bf b}_{n}}
\newcommand{\cbftilden}{\widetilde{\bf c}_{n}}
\newcommand{\lambdatilden}{\widetilde{\lambda}_{n}}
\newcommand{\mutilden}{\widetilde{\mu}_{n}}
\newcommand{\Eo}{\E_1}
\newcommand{\Et}{\E_2}
\newcommand{\Ei}{\E_i}

\newcommand{\vtilde}{\widetilde v}

\newcommand{\btildeoF}{\widetilde b_1^\F}
\newcommand{\bptildeoF}{b_{\ptildeEo}^{\F}}
\newcommand{\bptildetF}{b_{\ptildeEt}^{\F}}

\newcommand{\bptildeotF}{b_{\ptildeEo,\ptildeEt}^{\F}}

\newcommand{\ub}{\textbf{u}}
\newcommand{\vb}{\textbf{v}}

\newcommand{\fb}{\textbf{f}}
\newcommand{\zb}{\mathbf{0}}
\newcommand{\vbtilden}{\widetilde \vb_n}
\newcommand{\ubtilden}{\widetilde \ub_n}
\newcommand{\Deltab}{\bm{\Delta}}
\newcommand{\nablab}{\bm{\nabla}}

\newcommand{\Vb}{\textbf{V}}
\newcommand{\dive}{\operatorname{div}}
\newcommand{\diveh}{\operatorname{div}_\h}

\newcommand{\pE}{\p_\E}
\newcommand{\pEpo}{\p_{\E+1}}
\newcommand{\hpE}{h^{\pE}_{\E}}
\newcommand{\pEo}{\p_{\Eo}}
\newcommand{\pEt}{\p_{\Et}}
\newcommand{\ptildeE}{\ptilde_\E}
\newcommand{\ptildeEo}{\ptilde_{\Eo}}
\newcommand{\ptildeEt}{\ptilde_{\Et}}

\newcommand{\pF}{\p_\F}

\newcommand{\Ej}{\E_j}
\newcommand{\ptilde}{\widetilde \p}
\newcommand{\bptildeEF}{b^\F_{\ptildeE}}
\newcommand{\pbf}{\mathbf p}

\newcommand{\PiLo}{\Pi^{0,\F}_{L^\F_1}}
\newcommand{\PitildeLo}{\widetilde\Pi^{0,\F}_{L^\F_1}}

\newcommand{\DGp}{DG_\p}
\newcommand{\EDGp}{E_{\DGp}}
\newcommand{\CRp}{CR_\p}
\newcommand{\ECRp}{E_{\CRp}}
\newcommand{\tildeCRp}{\widetilde{CR}_\p}
\newcommand{\EtildeCRp}{E_{\tildeCRp}}

\newcommand{\uDG}{u_{\text{DG}}}
\newcommand{\vDG}{v_{\text{DG}}}
\newcommand{\etaDG}{\eta_{\text{DG}}}
\newcommand{\etaF}{\eta_\F}
\newcommand{\Abfn}{\mathbf A_n}
\newcommand{\Bbfn}{\mathbf B_n}
\newcommand{\Cbfn}{\mathbf C_n}
\newcommand{\Dbfn}{\mathbf D_n}
\newcommand{\Bbfnstar}{\Bbfn^*}
\newcommand{\Cbfnstar}{\Cbfn^*}
\newcommand{\Dbfnstar}{\Dbfn^*}
\newcommand{\zerobf}{\mathbf 0}
\newcommand{\ubfn}{\mathbf u_n}
\newcommand{\bbfn}{\mathbf b_n}
\newcommand{\sbfn}{\mathbf s_n}

\newcommand{\fbfn}{\mathbf f_n}
\newcommand{\Lcali}{\mathcal L^i}
\newcommand{\Lcalk}{\mathcal L^k}
\newcommand{\Lcalimo}{\mathcal L^{i-1}}
\newcommand{\Lcalz}{\mathcal{L}^0}
\newcommand{\taunzn}{\mathcal T_0}
\newcommand{\taunon}{\mathcal T_1}
\newcommand{\tauntw}{\mathcal T_2}
\newcommand{\EQtildeCRp}{\EtildeCRp^Q}
\newcommand{\EVtildeCRp}{\EtildeCRp^V}

\newcommand{\p}{p}
\newcommand{\h}{h}

\title{\footnotesize{New Crouzeix-Raviart elements
of even degree: theoretical aspects,
numerical performance,
and applications to the Stokes' equations}}
\author{\scriptsize A. Bressan\thanks{CNR-IMATI ``E. Magenes'',
Via Ferrata 5, 27100 Pavia, Italy {\tt andrea.bressan@imati.cnr.it}},
L. Mascotto\thanks{Dipartimento di Matematica e Applicazioni, Universit\`a di Milano-Bicocca, 20125 Milan, Italy, {\tt lorenzo.mascotto@unimib.it};
IMATI-CNR, 27100, Pavia, Italy;
Fakult\"at f\"ur Mathematik, Universit\"at Wien, 1090 Vienna, Austria, {\tt lorenzo.mascotto@univie.ac.at}},
M. Mosconi\thanks{Dipartimento di Matematica e Applicazioni, Universit\`a di Milano-Bicocca, 20125 Milan, Italy, {\tt m.mosconi@campus.unimib.it}}}
\date{}

\begin{document}

\maketitle

\begin{abstract}
\noindent 
We construct new Crouzeix-Raviart (CR) spaces of even degree~$\p$
in two dimensions
that are spanned by basis functions mimicking those for the odd degree case.
Compared to the standard CR gospel,
the present construction
allows for the use of nested
bases of increasing degree
and is particularly suited to design variable order CR methods.
We analyze a nonconforming discretization
of a two dimensional Poisson problem,
which requires a DG-type stabilization;
the employed stabilization parameter
is considerably smaller
than that needed in DG methods.
Numerical results are presented, which exhibit
the expected convergence rates
for the $\h$-, $\p$-, and $\h\p$-versions of the scheme.
We further investigate numerically
the behaviour of new even degree CR-type
discretizations of the Stokes' equations.


\end{abstract}

\section{Introduction} \label{section:introduction}

\paragraph*{State-of-the-art: the finite element framework.}
The lowest-order Crouzeix--Raviart (CR) element
was introduced in 1973 in~\cite{Crouzeix-Raviart:1973}
as an elementwise divergence free discretization
of the Stokes' equations
for space dimensions $d=2,3$;
see also \cite{Raviart-Thomas:1977}.
Fortin and Soulie~\cite{Fortin-Soulie:1983}
constructed explicit piecewise quadratic CR elements in 2D;
they showed that the second-order CR global space
is a standard conforming piecewise quadratic space
enriched with elemental nonconforming bubbles.
A CR scheme
with cubic approximation properties in 2D
was detailed in \cite[Example~5]{Crouzeix-Raviart:1973}
by augmenting the space with quartic bubble functions;
Crouzeix and Falk~\cite{Crouzeix-Falk:1989}
later introduced an optimal CR cubic element,
which did not require such bubbles.
Cha, Lee, and Lee~\cite{Cha-Lee-Lee:2000}
constructed quartic CR elements in 2D,
where fourth order nonconforming bubbles were employed,
with properties similar to those required by
Fortin and Soulie in the quadratic case.
In the subsequent years, the interest in developing higher (and general)
order CR elements raised a lot;
we refer to~\cite{Stoyan-Baran:2006,Carstensen-Sauter:2.2022,Ciarlet-Ciarlet-Sauter-Simian:2016}
for (essentially up to date) developments on this topic.

In the literature,
CR-type nonconforming global spaces
have been defined
\begin{itemize}
    \item \emph{implicitly}, as spaces of piecewise polynomials
    satisfying certain ``no-jump'' conditions at the interface
    between elements;
    \item \emph{explicitly}, as the span of the union of standard conforming finite element functions
    and additional nonconforming bubble functions.
\end{itemize}
While the former version is particularly useful on the theoretical level,
the latter appears to be very convenient in the implementation of the scheme.

It was clear since the very inception of the CR method \cite{Crouzeix-Raviart:1973}, see also \cite{Fortin-Soulie:1983},
that the \emph{explicit} definition
involves different types of basis functions
depending on the order of the scheme:
\begin{itemize}
    \item odd order CR elements are spanned by modal functions, yet
    removing the standard hat functions associated with the vertices of the mesh,
    and edge nonconforming bubbles;
    see Section~\ref{subsubsection:CR-standard-odd} below;
    \item even degree CR elements are spanned by modal functions
    and elemental nonconforming bubbles;
    see Section~\ref{subsubsection:CR-standard-even} below.
\end{itemize}
Apparently, this difference leads to some limitations:
\begin{itemize}
    \item in view of the $\p$-version of the scheme,
    the corresponding basis functions have not a hierarchical structure;
    \item variable order CR schemes does not contain
    the space of piecewise affine functions,
    cf. Subsection~\ref{subsection:variable-obvious-cannot-work}.
\end{itemize}

\paragraph*{State-of-the-art: beyond finite elements.} \label{paragraph:other-references}
More recent variants of CR-type elements,
which on occasion (typically on triangular meshes and
for low order of accuracy) coincide with CR elements,
are given by the nonconforming virtual element method
\cite{Beirao-Brezzi-Marini-Russo:2023, Ayuso-Manzini-Lipnikov:2016, Brezzi-Marini:2021}
and the hybrid high-order method \cite{DiPietro-Ern-Lemaire:2014, DiPietro-Ern-Lemaire:2016}.


\paragraph*{Goals of the paper.} \label{paragraph:goals}
In this work, we aim at facing the above issues.
In particular, we design novel even degree CR spaces
based on spanning sets of basis functions similar to those
employed in standard CR elements with odd degree.
This new approach
\begin{itemize}
    \item gives a hierarchical structure to the resulting $\p$-version CR method;
    \item allows for a simple and optimal design of variable order CR elements.
\end{itemize}
A moderate price to pay for these advantages
is that the standard broken
grad-grad bilinear form must be corrected
with DG-type terms in order to retain error estimates of the expected order;
this correction is not needed for the
well-posedness of the method
and, compared to its DG counterpart~\cite{DiPietro-Ern:2012},
has a much more moderate effect on the performance of the scheme,
as we shall investigate in practice.

As an interesting per se matter, we shall also check
the behaviour of new even degree CR-type discretizations
of the Stokes' equations on the numerical level;
this aspect is very much related to the recent
works \cite{Sauter:2023, Sauter-Torres:2023}.

\paragraph*{Outline of the remainder of this section.}
After recalling some functional and finite element tools
in Sections~\ref{subsection:functional-setting}
and~\ref{subsection:meshes},
we review the standard conforming modal elements in Section~\ref{subsection:modal}.
Section~\ref{subsection:CR-standard}
is devoted to the design of (implicit and explicit versions)
CR elements as in~\cite{Crouzeix-Raviart:1973}.
Eventually, the convergence analysis of standard CR methods
is detailed in Section~\ref{subsection:convergence-standard-CR}.

\paragraph*{Outline of the paper.}
New CR elements (and corresponding methods) of even degree
are the topic of Section~\ref{section:CR-even};
the theoretical results therein proven
are assessed on the numerical level in Section~\ref{section:numerics};
both the $\h$- and $\p$-versions are considered.
Variable order elements are detailed in Section~\ref{section:variable-degree};
there, we also check the performance of the $\h\p$-version of the scheme for the approximation
of corner singularities.
The behaviour of new even degree CR-type discretizations
of the Stokes' equations is the topic of Section~\ref{section:inf-sup}.
Conclusions are drawn in Section~\ref{section:conclusions}.

\subsection{Functional setting and model problem} \label{subsection:functional-setting}
We employ standard notation for differential operators:
$\Delta \cdot$ and $\nabla \cdot$ are the Laplacian and the gradient.

Consider a bounded, Lipschitz domain~$\Omega$ in~$\Rbb^2$
with boundary $\partial \Omega$.
Let $\normal{\Omega}$ and~$\hOmega$ be
its outward unit normal vector and diameter.
The space~$\Lebp$, $p$ in $[1,\infty)$,
is the space of Lebesgue measurable functions
that are $p$-integrable,
which we endow with the norm $\Norm{\cdot}_{L^p(\Omega)}$;
if $p=2$, we write $\Norm{\cdot}_{L^2(\Omega)} = \Norm{\cdot}_{0,\Omega}$.

Let $\boldalpha$ be a multi-index in $\mathbb{N}^2$,
$|\boldalpha|:=\sum_{i=1}^2\boldalpha_i$ be its length,
and $\weakder$ be the distributional derivative with partial derivative
of order~$\boldalpha_j$ along the direction $x_j$, $j=1,2$.
For~$s$ nonnegative integer and $p$ in $[1,\infty)$,
we consider the Sobolev spaces
\begin{equation*}
    W^{s,p}(\Omega):=\left\{v\in L^p(\Omega)\,|\,D^{\boldalpha}v\in L^p(\Omega)
                    \quad \forall \boldalpha\in\mathbb{N}^2,\; |\boldalpha|\leq s\right\}.
\end{equation*}
For a nonnegative integer~$s$,
we denote the Sobolev space $W^{s,2}(\Omega)$ by~$H^s(\Omega)$,
which we endow with standard inner product $(\cdot,\cdot)_s$,
seminorm~$\SemiNorm{\cdot}_{s,\Omega}$,
and norm~$\Norm{\cdot}_{s,\Omega}$;
in particular, we write
\[
\SemiNorm{\cdot}_{s,\Omega}^2
:=\sum_{\vert\boldalpha\vert=s}
\Norm{D^{\boldalpha}\cdot}_{0,\Omega}^2,
\qquad\qquad\qquad
\Norm{\cdot}_{s,\Omega} ^2
:= \sum_{\ell=0}^s \hOmega^{2(\ell-s)} 
\SemiNorm{\cdot}_{\ell,\Omega}^2 .
\]
Above, for $s = 0$, we are using the notation
$H^0(\Omega)=L^2(\Omega)$
and $\Norm{\cdot}_{0,\Omega} = \SemiNorm{\cdot}_{0,\Omega}$.
For a real~$s$, the Sobolev space~$H^s(\Omega)$ is defined via Riesz-Thorin interpolation.

A trace theorem holds true for Sobolev spaces \cite[Theorem~3.10]{Ern-Guermond:2021-A};
in particular, there exists a continuous operator $\trace$ from $H^s(\Omega)$ in~$L^2(\Omega)$
for $1/2 < s < 3/2$;
we denote the image of~$H^s(\Omega)$ through this operator
by~$H^{s-\frac12}(\partial \Omega)$.
Given~$g$ in $H^{s-\frac12}(\partial \Omega)$,
this result allows us to define the space
\[
H^s_g(\Omega) := \{v \in H^s(\Omega) \mid \gamma_\Omega(v) = g\}.
\]
For~$g=0$, this space coincides~\cite{Meyers-Serrin:1964}
with the closure of $C^{\infty}_0(\Omega)$ in~$\Hone$.

Sobolev spaces of negative order~$-s$, $s>0$, are defined by duality.
In particular, for $s=1$ and the $H^{-1}-H^1$ duality pairing $\langle \cdot, \cdot \rangle$,
we write
\[
H^{-1}(\Omega) = [H^{1}_0(\Omega)]',
\qquad\qquad\qquad
\Norm{\cdot}_{-1,\Omega}
= \sup_{w \in H^1_0(\Omega)} \frac{\langle \cdot, w \rangle}{\Norm{w}_{1,\Omega}}.
\]
\medskip

\noindent We consider the Poisson problem
\begin{equation}\label{strong-pois}
\begin{cases}
\text{find } u \text{ such that} \\
        -\Delta u  = f  & \text{in } \Omega \\
        u = 0           & \text{on } \partial \Omega.
\end{cases}
\end{equation}
Define the bilinear form $a:\Hone\times \Hone \to \Rbb$ as
\begin{equation*}
    \aparent{v}{w} :=  (\nabla v, \nabla w)_{0,\Omega} .
\end{equation*}
For~$f$ in~$H^{-1}(\Omega)$,
a weak formulation of \eqref{strong-pois} reads
\begin{equation} \label{weak:form}
\begin{cases}
\text{find } u \in H^1_0(\Omega) \text{ such that} \\ 
\aparent{u}{v} = \langle f,v \rangle & \forall v  \in H^1_0(\Omega). \\
\end{cases}
\end{equation}
The well-posedness of~\eqref{weak:form}
is standard and follows from Lax-Milgram's lemma.

\subsection{Meshes, and broken Sobolev and polynomial spaces}     \label{subsection:meshes}
We consider sequences $\{ \taun \}$ of conforming simplicial meshes in two dimensions,
and corresponding sequences of edges $\{ \Fcaln \}$ (resp. vertices $\{ \Vcaln \}$),
which we further split into boundary $\{ \FcalnB \}$ (resp. $\{ \VcalnB \}$)
and internal $\{ \FcalnI \}$ (resp. $\{ \VcalnI \}$ ) edges (resp. vertices).
For a given element~$\E$ in a mesh~$\taun$,
$\nbfE$ and~$\hE$ denote
its unit outward normal vector and diameter;
we denote the maximum diameter of a mesh $\taun$ by $\h:= \max_{K\in\mesh} \hE$.
For a given edge~$\F$ in~$\Fcaln$, $\nbfF$ and~$\hF$
denote a fixed once-and-for-all unit normal vector and its length;
on boundary edges, we fix $\nbfF = \nbfE{}_{|\F}$.
The sets of edges and vertices of a given element~$\E$ are~$\FcalE$ and ~$\VcalE$, respectively.
The vertices of~$\E$ are denoted by~$\nu_{i,K}$, $i=1,2,3$,
and the corresponding barycentric coordinates by $\lambdaEi$;
when convenient, we shall replace $\lambdaEi$
with~$\lambdaEF$ where~$\F$ is the edge
opposite to the $i$-th vertex,
and, when no confusion occurs,
we shall omit the subscript~$\E$
\footnote{Here, the barycentric coordinate associated
with the $i$-th vertex or the opposite edge $\F$
is the linear Lagrangian function
equal to~$1$ at that vertex and zero on $\F$.}.
For~$\F$ in~$\Fcaln$, we define edge patches
\begin{equation} \label{facet-patch}
\omegaF :=  \bigcup \{ \E \in \taun \mid \F \in \FcalE \}.
\end{equation}
We demand some regularity on the sequences of meshes:
the ratio between the maximum inradius and the minimum exradius
of all the elements in each mesh of the sequence is uniformly bounded;
in particular, there exists~$\shapereg$
larger than or equal to~$1$ such that
\begin{equation} \label{shape-regularity}
    \hE \le \shapereg \ \hF
    \qquad\qquad\qquad\qquad
    \forall \F \in \FcalE,
    \; \E \in \taun,
    \; \taun \in \{ \taun \}.
\end{equation}
With each mesh and nonnegative~$s$, we associate the broken Sobolev spaces
\[
H^s(\taun,\Omega)
:= \{ v \in L^2(\Omega) \mid v _{|\E} \in H^s(\E) \quad \forall \E \in \taun \};
\]
we further associate the broken gradient~$\nablah : H^1(\taun,\Omega) \to [L^2(\Omega)]^2$
defined as $(\nablah v)_{|\E} := \nabla (v_{|\E})$,
and the broken seminorm and norm
\begin{equation} \label{broken:sn-n}
\SemiNorm{v}\hoOm^2
:= (\nablah v, \nablah v)_{0,\Omega},
\qquad\qquad\qquad
\Norm{v}\hoOm^2
:=  \hOmega^{-2} \Norm{v}_{0,\Omega}^2
    + \SemiNorm{v}\hoOm^2.
\end{equation}
For each edge~$F$ in $\Fcaln$ and~$v$ in $H^{\frac12+\varepsilon}(\taun)$,
$\varepsilon$ positive,
we introduce the jump
\begin{equation}\label{eq:jumps}
\jump{v}_{\F} = \jump{v} :=
\begin{cases}
\traceEo (v)_{|\F} (\normal{\E_1}\cdot\nbfF) 
+ \traceEt (v)_{|\F} (\normal{\E_2}\cdot\nbfF) 
         & F\in \FcalnI,\, F = \partial \E_1\cap\partial \E_2 \\
\traceE(v)_{|\F}
        & F\in \FcalnB,\, F = \partial\Omega\cap \partial \E
\end{cases}
\end{equation}
and average operators
(taking values in $L^2(\F)$)
\[
\average{v}_\F = \average{v} =
\begin{cases}
\frac{1}{2}
( \traceEo(v) + \traceEt(v) )_{|\F}
& F\in \FcalnI,\, F = \partial \E_1\cap\partial \E_2 \\
\traceE(v)_{|\F}
& F\in \FcalnB,\, F = \partial\Omega\cap \partial \E.
    \end{cases}
\]
The vector version of these operators
is defined analogously
and the operator symbols are the same.

We denote the space of polynomials of order smaller than or equal to a nonnegative integer~$p$
over an element~$\E$ or a edge~$\F$ by $\Pbb_\p(\E)$ and~$\Pbb_\p(\F)$.
If $p$ is a negative integer,
then these two spaces are set to~$\{0\}$.
We further define
\begin{equation*}
\Pbbtaun 
:= \left\{ \qp \in L^2(\Omega) \Big \vert
    \qp{}_{|\E} \in \Pbb_\p (\E) \quad \forall K\in\mesh   \right\}
\end{equation*}
and
\begin{equation*}
\PbbB
:= \left\{ \qp \in L^2(\partial \Omega) \Big \vert
    {\qp}_{|\F} \in \Pbb_\p (\F) \quad \forall \F \in \FcalnB   \right\} .
\end{equation*}
We are now in a position to define nonconforming Sobolev spaces
of degree~$\p$ and order~$s$ larger than~$1/2$:
\begin{equation}\label{nc-Sobolev}
\Hsnctaunom{s}{\p}
:= \Big\{ v \in H^s(\taun,\Omega) \Big |
                        ( \jump{v}_\F,\qpmoF )_{0,\F} = 0 \quad \forall \qpmoF \in \Pbb_{\p-1}(\F),
                        \, \forall \F \in \FcalnI \Big\}.
\end{equation}
We can include ``boundary conditions'' in the nonconforming space above in a weak sense:
given~$g$ in $L^1(\F)$ for all boundary edges~$\F$, we set
\footnotesize\begin{equation}\label{nc-Sobolev-BCs}
H^{s,nc}_{\p,g} (\taun,\Omega)
:= \Big\{ v \in \Hsnctaunom{s}{\p} 
    \Big | ( \jump{v}_\F, \qpmoF )_{0,\F} 
    = ( g , \qpmoF)_{0,\F}
    \quad \forall \qpmoF \in \Pbb_{\p-1}(\F),
    \, \forall \F \in \FcalnB \Big\}.
\end{equation} \normalsize
For positive~$a$ and~$b$,
we shall denote with~$a \lesssim b$
the existence of a positive constant~$C$
only depending on the domain~$\Omega$,
the constant~$\rho$ in~\eqref{shape-regularity},
and on occasions on the polynomial degree,
such that~$a \leq C\, b$.

A generalized Poincar\'e--Friedrichs inequality
holds true~\cite{Brenner:2003}:
\begin{equation} \label{Poincare-Friedrichs}
     \Norm{v}_{0,\Omega} \lesssim \SemiNorm{v} \hoOm
     \qquad\qquad\qquad
     \forall v \in H^{1,nc}_{1,0}(\taun,\Omega).
\end{equation}
As a matter of notation, given $\alpha,\beta>-1$,
$\Lp(\cdot)$ and $J_\p^{(\alpha,\beta)}(\cdot)$ denote
the univariate Legendre and Jacobi polynomials of degree $p$ over~$\Ihat:= [-1,1]$.

\subsection{Conforming (modal) finite elements} \label{subsection:modal}
We introduce the standard conforming finite element space of order~$\p$
\begin{equation} \label{Xn}
\Xn := \{ \vn \in \Hone \mid {\vn}_{|\E} \in \Pbb_{p}(\E) \quad \forall \E \in \taun \}.
\end{equation}
An explicit basis of this space
is detailed in~\cite{Szabo-Babuska:1991}.
Here, we recall the basis functions on the
reference element~$\Ehat$
with vertices $(-1,0)$, $(1,0)$, and~$(0,\sqrt 3)$,
and, with an abuse of notation, barycentric coordinates
$\widehat\lambda_{\Ehat,\Fhat_i}$, $\Fhat_i$ in~$\FcalEhat$
for $i=1,2,3$.
On~$\Ehat$, we distinguish three different
types of basis functions:
\begin{itemize}
\item given~$\nuhat_m$, $m=1,2,3$, the vertices of~$\Ehat$, the hat functions
    \begin{equation*}
        \phihatjVhati{1} = \frac12 \left( 1-x-\frac{y}{\sqrt{3}}\right),
        \qquad\qquad
        \phihatjVhati{2} = \frac12 \left( 1+x-\frac{y}{\sqrt{3}}\right),
        \qquad\qquad
        \phihatjVhati{3} = \frac{y}{\sqrt 3};    
    \end{equation*}
\item given the edges~$\Fhati$, $i=1,2,3$, of~$\Ehat$,
    the $\dim(\Pbb_{\p-2}(\Fhat))$ edge bubble functions
    \small\begin{equation} \label{modal-facet}
    \begin{split}
    & \phihatjFhati := \frac{8\sqrt{4j+2}}{j(j+1)}
    \lambdahatFhati{i\bmod{3}+1}\lambdahatFhati{(i+1)\bmod{3}+1}
    \Lj'\left(-\lambdahatFhati{i\bmod{3}+1}+\lambdahatFhati{(i+1)\bmod{3}+1}\right) \\
    & \qquad\qquad\qquad \forall i=1,2,3 ,
        \;   j = 1,\dots,\p-1;
    \end{split}  
    \end{equation} \normalsize
\item given the natural bijection
    between $\ell$ in $\{1,\dots,\dim\Pbb_{\p-3}(\Ehat)\}$
    and $(r_1,r_2)$ in $\{ 0,\dots,\p-3\}^2$
    with $r_1+r_2\le \p-3$,
    the $\dim(\Pbb_{\p-3}(\Ehat))$ elemental bubble functions
    \begin{equation} \label{modal-element}
    \begin{split}
    & \widehat\varphi_{\ell}^{\widehat\E} 
    = \widehat\varphi_{r_1,r_2}^{\widehat\E}
    := \lambdahatFhati{1}\lambdahatFhati{2}\lambdahatFhati{3} L_{r_1}(\lambdahatFhati{2}-\lambdahatFhati{1})L_{r_2}(2\lambdahatFhati{3}-1)\\
    & \qquad\qquad\qquad\qquad
            \forall  r_1, r_2 \text{ with } r_1+r_2 = 0,\dots,p-3.
    \end{split}
    \end{equation}
\end{itemize}

\subsection{Arbitrary order, standard Crouzeix-Raviart elements in 2D} \label{subsection:CR-standard}
We introduce the Crouzeix-Raviart (CR) space
\cite{Crouzeix-Raviart:1973, Raviart-Thomas:1977}
of positive integer order~$\p$:
\begin{equation}\label{implicit-CR}
        \Vnim := \{ \vn \in \Hsnctaunom{1}{\p} |\,
                                {\vn}_{|K}\in\Pbb_p(K)
                                \quad \forall \E \in \taun \}.
\end{equation}
In two dimensions, this spaces coincides~\cite{Ainsworth-Rankin:2008} with the following space:
\small\begin{equation*}
   \begin{split}
    \Vnim = \{ \vn \in \Pbb_p(\mesh)  \,|&\,
     \vn \text{ is continuous at the Gau\ss--Legendre nodes of order~$\p$ on each } F\in\FcalnI\}.
\end{split}
\end{equation*} \normalsize
We shall refer to these spaces, which are nonconforming in the sense
of~\cite{Strang:1972}, as the \emph{implicit} CR spaces of order~$\p$.
It is known~\cite{Baran-Stoyan:2007} that unisolvent sets
of degrees of freedom for the space~$\Vnim$
are different for odd and even~$\p$.

An alternative way \cite{Stoyan-Baran:2006, Baran-Stoyan:2007}
to construct nonconforming spaces \emph{\'a la} Crouzeix-Raviart
consists in defining \emph{explicit} spaces, i.e., spans
of $H^1$ conforming finite element functions
\emph{plus} extra nonconforming bubbles.
In Sections~\ref{subsubsection:CR-standard-odd} and~\ref{subsubsection:CR-standard-even} below,
we provide details on such bubble functions in two dimensions
for odd and even degrees, respectively;
show that the \emph{implicit} and \emph{explicit} constructions
deliver exactly the same spaces;
exhibit unisolvent sets of degrees of freedom.

\subsubsection{Standard Crouzeix-Raviart elements (odd degree)}
\label{subsubsection:CR-standard-odd}

Let~$p$ be an odd positive integer.
Given~$\E$ an element, consider for all~$\F$ in~$\FcalE$
the mapped Legendre polynomials~$\LjF$ on~$\F$, $j$ nonnegative integer,
and a basis $\{\mboldbetaE\}_{\vert\boldbeta\vert=0}^{\p-3}$ of $\Pbb_{\p-3}(\E)$
consisting of elements that are invariant with respect to dilations and translations.
We introduce a set of linear functionals:
given~$v$ in~$W^{1,1}(\E)$,
which admits trace in $L^1(\F)$
for all $\F$ in $\FcalE$
\cite[Theorem 3.10]{Ern-Guermond:2021-A}
\footnote{The integrals in~\eqref{dof-moment:odd} have to be understood as duality pairings in $L^1-L^\infty$.
Similar comments will be omitted in the sequel.},
\begin{subequations}  \label{dof-moment:odd}
\begin{align}
    &|\F|^{-1} (v,\LjF)_{0,\F}
    & \forall j = 0, \dots, \p-1, \quad\forall \F\in \FcalE; \label{dof-moment-edge:odd}\\
    & |\E|^{-1} (v,\mboldbetaE)_{0,\E}
    &     \forall \vert \boldbeta\vert = 0, \dots, \p-3 \label{dof-moment-bulk:odd}.
\end{align}
\end{subequations}
A crucial property of the above linear functionals
is detailed in the next result.

\begin{lemma}\label{lemma:unisolv-odd}
Given~$\E$ in~$\taun$ and an odd~$\p$,
the linear functionals \eqref{dof-moment:odd}
are a unisolvent set of degrees of freedom
for~$\Pbb_{\p}(\E)$.
\end{lemma}
\begin{proof}
Let $\vn$ be in $\Pbb_\p(\E)$
such that the moments \eqref{dof-moment:odd} are zero.
The number of linear functionals in \eqref{dof-moment:odd} is equal to
$\dim (\Pbb_{\p}(\E))$;
therefore, it suffices to prove that $\vn$ is identically zero over~$\E$.

Given an edge~$\F$ in~$\FcalE$,
using that the edge moments \eqref{dof-moment-edge:odd} are zero,
the restriction of $\vn$ on~$\F$ belongs to $\Pbb_{\p}(\F)$
and is orthogonal to $\Pbb_{\p-1}(\F)$.
Hence, there exists $c$ in $\Rbb$ such that
${\vn}_{|\F}(\xi) = c\, L^\F_p(\xi)$ for all $\F$ in~$\FcalE$.
Recall~\cite[eq. (3.175)]{Shen-Tang-Wang:2011}
that $L_\p(-1)=(-1)^\p = -1$ and~$L_\p(1)=1$.
Since $\p$ is odd and ${\vn}_{|\partial \E}$ belongs to $\mathcal C^0(\partial \E)$,
the following identities are valid:
\begin{equation} \label{trick-odd-case}
    \vn(\nu_1) = -\vn(\nu_2) = \vn(\nu_3) = -\vn(\nu_1).
\end{equation}
Therefore, $\vn(\nu_1)=\vn(\nu_2)=\vn(\nu_3)=0$
and there exists $q_{\p-3}^\E$ in $\Pbb_{\p-3}(\E)$
such that $\vn = \lambda_1 \lambda_2 \lambda_3 q_{\p-3}^\E.$
Using that the elemental moments \eqref{dof-moment-bulk:odd}
are zero as well, we infer
$( \lambda_1 \lambda_2 \lambda_3 (q_{\p-3}^\E)$, $q_{\p-3}^\E)_{0,\E}$ $= 0.$
Since the product of the barycentric coordinates in~$\E$ is positive,
we deduce that~$q_{\p-3}^\E$ vanishes in~$\E$.
\end{proof}
\begin{remark} \label{remark:GL}
An alternative avenue to define edges
degrees of freedom consists
in replacing~\eqref{dof-moment-edge:odd}
by point evaluations
at the $\p$ Gau\ss-Legendre nodes \cite[Lemma 5]{Raviart-Thomas:1977};
from an historic point of view, 
this was the original approach in~\cite{Crouzeix-Raviart:1973}.
\end{remark}

Next, we discuss an alternative definition
for CR elements of odd degree.
With each edge~$\F$ in~$\Fcaln$,
we associate the edge nonconforming bubble function~\cite{Stoyan-Baran:2006}
\begin{equation}\label{nc-edge:functions}
\bpF :=
\begin{cases}
      L_p\left( 1-2 \lambdaEF\right)
      \qquad\qquad
      &\forall \E\in\taun,\; \E \subset \omegaF  \\
      0 
      &\text{otherwise}.
\end{cases}
\end{equation}
Under the notation in Section~\ref{subsection:modal}, we define
\begin{equation}\label{explicit-CR:odd}
    \Vnex := \Span\left(                
                \bpF \text{ and }  \{\phijF\}_{j=1}^{\p-1}\quad
                \forall \F \subset \Fcaln ;\quad
                \{\phiellE\}_{\ell=1}^{(\p-1)(\p-2)/2} 
                \quad \forall \E \in \taun   \right).
\end{equation}
The bubble function~$\bpF$ does not belong to~$\Xn$ in~\eqref{Xn},
whence the name ``\emph{nonconforming}''.
As discussed in~\cite{Carstensen-Sauter:2022,Ciarlet-Ciarlet-Sauter-Simian:2016,Stoyan-Baran:2006},
the functions spanning the right-hand side of~\eqref{explicit-CR:odd}
are linearly independent.
This fact allows us to define an explicit basis
for~$\Vnex$; the corresponding degrees
of freedom are the coefficients
of the expansion with respect to that basis.

The next result investigates the relation
between the spaces~$\Vnim$ and~$\Vnex$
for odd~$\p$.

\begin{lemma} \label{lemma:equivalent-definition-spaces:odd}
The spaces~$\Vnim$ and~$\Vnex$ defined
in~\eqref{implicit-CR} and~\eqref{explicit-CR:even} coincide
for odd~$\p$.
\end{lemma}
\begin{proof}
Let $\numedt$, $\numedi$, and $\numel$ denote the number
of total edges, internal edges, and elements of~$\taun$, respectively.
We split the proof in two steps.

\paragraph*{$\Vnex$ is contained in~$\Vnim$.}
The space~$\Vnex$ consists of modal basis functions,
i.e., functions that are $H^1$ conforming
and therefore contained in~$\Vnim$,
and edge bubble functions as in~\eqref{nc-edge:functions}.
Such bubble functions
have vanishing jumps of order up to~$\p-1$
over all the edges in~$\Fcaln$.

To see this, we fix a generic~$\F$ in~$\Fcaln$ and~$\bpF$ as in~\eqref{nc-edge:functions}.
We have that~$\bpF$ is continuous
over the interior of~$\omegaF$,
whence it has vanishing jumps over~$\F$.
On the edges~$\widetilde \F$ lying on~$\partial \omegaF$,
$\bpF{}_{|\widetilde\F}$ is the univariate
Legendre polynomial of order~$\p$
starting from the vertex opposite to~$\F$;
in particular, its moments with univariate polynomials
of order~$\p-1$ are zero over~$\widetilde \F$.
On the other edges of the mesh,
the bubble is zero by definition, whence it has zero jumps.

\paragraph*{$\Vnim$ and~$\Vnex$ have the same dimension.}
The dimension of~$\Vnim$
is the dimension of piecewise polynomials
of maximum order~$\p$ over~$\taun$,
minus the total number of ``no jump''
conditions
over the internal edges:
\begin{equation} \label{dim:Vnim}
\dim( \Vnim )
= \dim( \Pbbtaun ) - \p \numedi
= \dim( \Pbb_{\p}(\E)) \numel - \p \numedi.
\end{equation}
On the other hand, the dimension of~$\Vnex$ is given by
the number of total edges
(for the nonconforming bubbles)
plus the number of modal basis functions
excluding those associated with the vertices:
\begin{equation} \label{dim:Vnex}
\dim( \Vnex ) 
= \numedt
  + \dim( \Pbb_{\p-2}(\F)) \numedt
  + \dim( \Pbb_{\p-3}(\E)) \numel 
= \p \numedt
   + \dim( \Pbb_{\p-3}(\E)) \numel .
\end{equation}
Recall the Euler's formula
\begin{equation} \label{euler-1}
    \numedi+\numedt = 3\numel.
\end{equation}
Subtracting the identities in~\eqref{dim:Vnim} and~\eqref{dim:Vnex}, we deduce
\[
\frac{(p+2)(p+1)-(p-1)(p-2)}{2}\numel - p(\numedi+\numedt)
\overset{\eqref{euler-1}}{=} 0.
\]
\end{proof}

Lemma~\ref{lemma:equivalent-definition-spaces:odd}
entails that two different sets of degrees of
freedom are at disposal.
A set of global degrees of freedom
for Crouzeix-Raviart spaces
is obtained by a coupling of the degrees
of freedom in~\eqref{dof-moment-edge:odd};
alternatively, we pick the coefficients in the expansion
with respect to the basis in~\eqref{explicit-CR:odd}.
The former is used on the theoretical level,
e.g., while deriving interpolation estimates in CR spaces,
see Remark~\ref{remark:interpolation} below;
the latter is typically used in the practice
for the implementation of the method.

\begin{remark} \label{remark:3D-case}
Throughout, we shall be proving results
similar to Lemma~\ref{lemma:equivalent-definition-spaces:even}.
Another reason why we distinguish the implicit
and explicit spaces
is that in higher dimension
the two definitions do not coincide in general;
cf., e.g., \cite{Sauter-Torres:2023}.
\end{remark}

\subsubsection{Standard Crouzeix-Raviart elements with even degree} \label{subsubsection:CR-standard-even}
Proceeding as in Section~\ref{subsubsection:CR-standard-odd} for even~$\p$,
we are not able to show unisolvence
of the degrees of freedom~\eqref{dof-moment:odd}:
the proof of Lemma~\ref{lemma:unisolv-odd} fails
since display~\eqref{trick-odd-case} would read
\begin{equation} \label{trick-even-case}
\vn(\nu_1) = \vn(\nu_2) = \vn(\nu_3) = \vn(\nu_1),
\end{equation}
which does not imply that~$\vn$ is zero
at the vertices of~$\E$;
furthermore,
using Gau\ss-Legendre-type degrees of freedom
as those discussed in Remark~\ref{remark:GL}
would lead to a linear dependence condition
as described in~\cite[pp. 506-507]{Fortin-Soulie:1983}
for the case $\p=2$.

We recall some details here;
the general even degree case presents similar limitations.
Given a triangle~$\E$,
we denote its unit tangential vector
and the corresponding tangential derivative
by~$\tauboldE$ and~$\partial_{\tauboldE}$.
Let $\{g_j^i\}_{i=1}^2$ and $m_j$ 
denote the Gau\ss--Legendre quadrature points
and the midpoints
of the edge~$F_j$, $j=1,2,3$, respectively. 

Given a quadratic polynomial~$\qtw$,
$\partial_{\tauboldE} \qtw{}_{|\partial\E}$ is piecewise linear:
the integral over the edge~$\F_j$
of $\partial_{\tauboldE} \qtw$
is equal to the evaluation at the midpoint~$m_j$;
in turn, this can be expressed as the difference
of the evaluation at the two Gau\ss-Legendre nodes,
which are symmetric with respect to~$m_j$.
Moreover, since $\qtw$ is smooth over~$\partial\E$,
we have
\begin{equation*}
    0=\int_{\partial \E} \partial_{\tauboldE} \qtw \dx{x}
    = \sum_{j=1}^3\mis{F_j} \partial_{\tauboldE} \qtw(m_j)
    = \sum_{j=1}^3\mis{F_j} \frac{\qtw(g_j^2)-\qtw(g_j^1)}{\mis{F_j}} \sqrt{3}.
\end{equation*}
Hence,
the following relation holds true:
\begin{equation*} 
    \qtw(g_1^2)-\qtw(g_1^2)+\qtw(g_2^2)-\qtw(g_2^1)+\qtw(g_3^2)-\qtw(g_3^1) = 0.
\end{equation*}
This is a linear dependence condition involving all the
degrees of freedom.
Moreover, it is possible to check~\cite{Fortin-Soulie:1983} that the quadratic polynomial
$2 - 3 (\lambda_1^2+\lambda_2^2+\lambda_3^2)$
has vanishing degrees freedom but is not zero in~$\E$.
Consequently, the choice of the degrees of freedom in~\eqref{dof-moment:odd}
is not suited to describe even degree CR spaces.
\medskip

In order to establish a set of unisolvent degrees of freedom,
the even counterpart of the explicit space in~\eqref{explicit-CR:odd}
comes to the rescue.
With each element~$\E$ in~$\taun$,
we associate the elemental nonconforming bubble function~\cite{Stoyan-Baran:2006}
\begin{equation}\label{nc-elemental:bubble}
    {b}_{p}^{\E} := 
    \begin{cases}
      \frac{1}{2}\big(-1+\sum\limits_{i=1}^{3}L_p(1-2\lambda_i)\big)
      \qquad\qquad
      &\text{on } \E\\
      0
      &\text{otherwise}.
    \end{cases}
\end{equation}
Following \cite{Ciarlet-Ciarlet-Sauter-Simian:2016,Stoyan-Baran:2006,Carstensen-Sauter:2022},
under the notation in Section~\ref{subsection:modal}, we define
\begin{equation} \label{explicit-CR:even}
   \Vnex := \Span\left(
                \varphi^\nu \quad \forall \nu \subset \Vcaln ;
                \quad
                \{\phijF\}_{j=1}^{\p-1}\quad
                \forall \F \subset \Fcaln;
                \quad
                \bpE \text{ and } \{\phiellE\}_{\ell=1}^{(\p-2)(\p-1)/2} 
                \quad \forall \E \in \taun \right).
\end{equation}
As discussed in \cite[Proposition 1]{Fortin-Soulie:1983} for the case $\p=2$,
these functions are linearly dependent:
there exists precisely one elemental nonconforming bubble
that can be written as a linear combination
of the other functions generating the space;
some details are provided in Lemma~\ref{lemma:equivalent-definition-spaces:even} below.

\begin{lemma} \label{lemma:equivalent-definition-spaces:even}
The spaces~$\Vnim$ and~$\Vnex$ defined
in~\eqref{implicit-CR} and~\eqref{explicit-CR:even} coincide for even~$\p$.
\end{lemma}
\begin{proof}
Introduce $\Bn := \Span\{\bpE \mid \E \in \taun\}$.
Let $\numedt$, $\numedi$, $\numver$, and $\numel$ denote the number
of total edges, internal edges, vertices, and elements of~$\taun$, respectively.
We split the proof in two steps.

\paragraph*{$\Vnex$ is contained in~$\Vnim$.}
The space $\Vnex$ is defined as $\Xn + \Bn$, i.e.,
it consists of modal basis functions,
which are $H^1$ conforming and therefore contained in $\Vnim$,
and elemental nonconforming bubbles as in \eqref{nc-elemental:bubble}.
Such bubble functions have vanishing jump moments
with respect to polynomials of order~$\p-1$
over the edges in~$\Fcaln$.

To see this,
we fix a generic~$\F$ in~$\FcalnI$
and one of the two elements~$\E$ in~$\omegaF$.
The restriction of~$\bpE$ on~$\F$
is given by the sum of four terms,
see~\eqref{nc-elemental:bubble}:
two of them are (scaled) univariate Legendre
polynomials of degree~$\p$,
which have vanishing moments up to order~$\p-1$;
the sum of the other two vanishes
since the restriction of the remaining
Legendre polynomial over~$\F$ is equal to~$1$.

\paragraph*{$\Vnim$ and~$\Vnex$ have the same dimension.}
As shown in the proof of Lemma~\ref{lemma:equivalent-definition-spaces:odd},
the dimension of~$\Vnim$ is given by
\[
\dim( \Vnim ) 
= \dim (\Pbbtaun) - \p \numedi
= \dim (\Pbb_{\p}(\E)) \numel - \p \numedi.
\]
As for the dimension of~$\Vnex$,
we start by investigating the dimensions of~$\Xn$
and~$\Bn$, separately.
The former is given by the number of total
modal basis functions:
\[
\dim( \Xn )
= \dim (\Pbb_{\p-3}(\E)) \numel
  + \dim (\Pbb_{\p-2}(\F)) \numedt
  + \numver;
\]
the latter by the number of elements:
\[
\dim( \Bn ) = \numel.
\]
Based on the Euler's formulas~\eqref{euler-1} and $\numel-\numedt+\numver = 1$,
the dimensions of $\Vnim$, $\Xn$, and~$\Bn$
satisfy the following relation:
\begin{equation} \label{relation:dim}
\dim(\Xn) + \dim(\Bn) = \dim(\Vnim) + 1.
\end{equation}
As elemental bubbles of neighbouring elements
have the same restriction on the interface,
we have 
\begin{equation} \label{dim-intersection}
\Xn \cap \Bn =
\Span \Big\{\sum_{\E\in\taun} \bpE \Big\},
\qquad\qquad\qquad
\dim(\Xn\cap\Bn) =1.
\end{equation}
We deduce
\[
\begin{split}
\dim (\Vnim) + 1
& \overset{\eqref{relation:dim}}{=}
        \dim(\Xn) + \dim(\Bn)
= \dim (\Xn + \Bn) + \dim (\Xn \cap \Bn) 
\overset{\eqref{dim-intersection}}{=} \dim (\Vnex) + 1.
\end{split}
\]
\end{proof}

\begin{remark} \label{remark:explicit-basis}
A consequence of Lemma~\ref{lemma:equivalent-definition-spaces:even}
is that an explicit basis for~$\Vnim$ and~$\Vnex$
is given by the generating functions 
in~\eqref{explicit-CR:even},
after removing one elemental nonconforming bubble;
the corresponding degrees of freedom are the coefficients
in the expansion with respect to that basis.
Differently from the case~$\p$ odd, we have not at disposal
a set of degrees of freedom as those in~\eqref{dof-moment:odd}.
As such, it is not apparent how to derive
interpolation estimates with respect
to proposed type of basis.
In Section~\ref{section:CR-even} below,
we shall propose an alternative CR construction,
which allows us to circumvent this issue.
\end{remark}

\subsection{Convergence for standard CR elements} \label{subsection:convergence-standard-CR}
As we proved that the implicit and explicit
CR spaces are the same,
see Lemmas~\ref{lemma:equivalent-definition-spaces:odd}
and~\ref{lemma:equivalent-definition-spaces:even},
we henceforth write~$\Vn$
to denote either of the two.
Moreover, given $g$ in $L^1(\partial \Omega)$,
we write
\[
\Vng := \Hsncgtaunom{1}{\p} \cap \Vn. 
\]
Define the bilinear form $\an: [\Hone + \Vn]^2 \to \Rbb$ as
\begin{equation}\label{broken-bilienarform}
    \anparent{\vn}{\wn}
    :=
    (\gradh \vn, \gradh \wn)_{0,\Omega}.
\end{equation}
We review~\cite{Brenner:2015} the standard analysis
of the following method:
\begin{equation} \label{method-standard}
\begin{cases}
\text{find } \un \in \Vnzero \text{ such that}\\ 
\anparent{\un}{\vn}  =  ( f,\vn )_{0,\Omega} & \forall \vn \in \Vnzero.
\end{cases}
\end{equation}
The treatment of inhomogeneous Dirichlet boundary conditions is discussed
in Section~\ref{section:numerics} below.

The bilinear form in~\eqref{broken-bilienarform}
is coercive and continuous with constants~$1$
with respect to the broken seminorm in~\eqref{broken:sn-n};
in turn, that seminorm is a norm on~$\Vnzero$ due to~\eqref{Poincare-Friedrichs}.
Therefore, method~\eqref{method-standard} is well-posed
and the following convergence result holds true;
its proof is standard~\cite{Brenner:2015}.
We report some details as they will be needed
in the analysis of new CR-type methods,
see Section~\ref{section:CR-even}.

\begin{proposition} \label{proposition:standard-convergence}
Let~$u$ and~$\un$ be the solutions to~\eqref{weak:form} and~\eqref{method-standard}.
Then, we have
\begin{equation}\label{strang}
    \SemiNorm{u-\un}_{1,\h; \Omega}
    \le
    \inf_{\vn \in \Vnzero} \SemiNorm{u-\vn}_{1,\h; \Omega} 
    + \sup_{\vn \in \Vnzero} \frac{|\anparent{u}{\vn}-( f,\vn )_{0,\Omega}|}{\SemiNorm{\vn}_{1,\h; \Omega}}.
\end{equation}
If~$u$ belongs to $H^{s} (\Omega)$,
$s$ larger than~$3/2$,
and~$\p$ is smaller than or equal to~$s$,
then we have
\begin{equation} \label{standard-convergence}
\SemiNorm{u-\un}_{1,\h; \Omega}
\lesssim \h^\p \SemiNorm{u}_{\p+1,\Omega}. 
\end{equation}
\end{proposition}
\begin{proof}
Inequality~\eqref{strang} is second Strang's lemma~\cite{Strang:1972,Brenner:2015}.  
As for estimate~\eqref{standard-convergence}, we bound the two terms
on the right-hand side of~\eqref{strang} separately.

As for the first one,
the best error in CR spaces
is smaller than the best error in conforming Lagrangian spaces:
the convergence follows from standard
polynomial interpolation theory~\cite{Brenner-Scott:2008};
the $H^s(\Omega)$ regularity, $s$ larger than $3/2$,
suffices to this purpose.

As for the second term, for all~$\F$ in~$\Fcaln$,
introduce $\PiltwoF{\p-1}$ as the orthogonal projector
with respect to the $L^2(\F)$ inner product
into the space~$\Pbb_{\p-1}(\F)$ of polynomial vector fields.
Denote the elements in~$\omegaF$ by~$\E_i$, $i=1,2$.
An integration by parts, definitions~\eqref{eq:jumps} and~\eqref{nc-Sobolev},
formulation~\eqref{strong-pois},
the Cauchy--Schwarz inequality,
and the assumption that $u$ belongs to $H^{s}(\Omega)$, $s$ larger than~$3/2$, yield,
for all $\qpEo$, $\qpEtw$, $\qzEo$, and $\qzEtw$
in $\Pbb_{\p}(\E_1)$, $\Pbb_{\p}(\E_2)$, $\Pbb_{0}(\E_1)$, and~$\Pbb_{0}(\E_2)$,
\begin{equation}\label{standard-convergence:1}
\begin{split}
\anparent{u}{\vn} -(f,\vn)_{0,\Omega}
& =
\sum_{\E \in \taun} (\normal{\E}\cdot\nabla u ,\vn)_{0,\partial\E} 
= \sum_{\F \in \Fcaln} (\normal{\F}\cdot \nabla u, \jump{\vn})_{0,\F} \\
&= \sum_{\F \in \Fcaln} \left(\normal{\F}\cdot \nabla u 
                - \PiltwoF{p-1} \normal{\F}\cdot \nabla u,
        \jump{\vn} - \PiltwoF{p-1} \jump{\vn} \right)_{0,\F}. 
\end{split}                                
\end{equation}
\normalsize
For interior edges, we have
\[
\begin{split}
& \left(\normal{\F}\cdot \nabla u 
                - \PiltwoF{p-1} \normal{\F}\cdot \nabla u,
        \jump{\vn} - \PiltwoF{p-1} \jump{\vn} \right)_{0,\F} \\
& \le 
    \Norm{\frac12 \normal{\E_1}\cdot \nabla (u_{|\E_1} - \qpEo)
          - \frac12 \normal{\E_2}\cdot \nabla (u_{|\E_2} - \qpEtw)}_{0,\F}
    \Norm{\frac12 (\vn{}_{|\E_1} - \qzEo) + \frac12 (\vn{}_{|\E_2} - \qzEtw)}_{0,\F}.
\end{split}
\]
The multiplicative trace inequality, the Poincar\'e inequality,
and choosing the polynomial vector fields above
as suitable elementwise $L^2$ projections
entail that the two terms on the right-hand side of~\eqref{standard-convergence:1}
satisfy, for $i=1,2$ and positive~$\varepsilon$,
\[
\begin{split}
\Norm{\normal{\E_i}\cdot \nabla (u_{|\E_1} - \qpEi)}_{0,\F} 
&\lesssim \hF^{-\frac12} \Norm{\gradh(u-\qpEi)}_{0,\Ei}
     + \hF^{\varepsilon}\SemiNorm{u-\qpEi}_{\frac32+\varepsilon,\Ei} ; \\
\Norm{\vn{}_{|\E_i} - \qzEi}_{0,\F}
&\lesssim \hF^{-\frac12} \Norm{\vn - \qzEi}_{0,\Ei}
            + \hF^{\frac12} \SemiNorm{\vn}_{1,\Ei}
    \lesssim h^{\frac12} \SemiNorm{\vn}_{1,\Ei}.
\end{split}
\]
Analogous estimates can be proven for boundary edges.
Combining the above displays with~\eqref{standard-convergence:1},
and standard polynomial approximation results yield the assertion.
\end{proof}

\begin{remark} \label{remark:interpolation}
For odd degree~$\p$,
given~$v$ in $W^{1,1}(\Omega)$,
we can define~$\vI$ in~$\Vn$
by fixing its degrees of freedom~\eqref{dof-moment:odd}
equal to those of~$v$.
A standard Bramble-Hilbert argument
entails optimal interpolation estimates
under regularity assumptions
that are milder than those needed for pointwise interpolation operators
as used in Lemma~\ref{proposition:standard-convergence}.
\end{remark}

\begin{remark}
\label{remark:shortcomings}
One of the most striking features of CR elements
is that the structure of the explicit
basis functions
is very different in the odd and even degree cases;
see~\eqref{explicit-CR:odd} and~\eqref{explicit-CR:even}.
On the one hand, in the light of Remark~\ref{remark:interpolation},
this renders the derivation of interpolation estimates
in the $\p$-version of the method less immediate
than for conforming and DG methods.
On the other hand,
in the assembling process of the matrix stemming from method~\eqref{method-standard}
for increasing~$\p$,
it is not feasible to use the same indexing strategy for even and odd degrees
since the types of basis functions under consideration change.
In the odd case, modal vertex functions are not considered
and edge nonconforming bubble functions are used;
in the even case, modal vertex functions are considered
and elemental nonconforming bubble functions are used.
\end{remark}

\section{New Crouzeix-Raviart elements of even degree} \label{section:CR-even}
Here, we propose a different construction of even degree CR elements,
which allows for circumventing some of the shortcomings discussed in
Remark~\ref{remark:shortcomings}.
In particular, in Section~\ref{subsection:new-CR-spaces},
we define a new implicit CR-type space of even degree;
show an explicit construction of a corresponding
spanning set of functions;
identify a set of unisolvent DoFs
analogous to that introduced for the standard odd case.
Moreover, in Section~\ref{subsection:optimal-convergence},
we propose a stabilized, arbitrary order CR-type method,
which is suitable for $\h\p$-refinements,
and prove its optimal convergence.

\subsection{edge-bubble-type even degree Crouzeix-Raviart spaces} \label{subsection:new-CR-spaces}
Let~$\p$ be an even positive integer.
We define a modification
of the Sobolev nonconforming space in~\eqref{nc-Sobolev}
as follows:
\[
    \Htildesnctaunom{1}{\p} := \left\{ \vtilde \in H^1(\taun,\Omega)  \,\Big | \,
    (\jump{\vtilde},L^\F_j )_{0,\F} = 0 \quad
    \forall j = 0,\dots,\p-2,\p, \quad
    \forall \F\in\FcalnI  \right\}.
\]
Based on this, we introduce a variant of the space~$\Vnim$ in~\eqref{implicit-CR}:
\begin{equation}\label{new:implicit-CR}
\Vnimtilde := \left\{ \vtilden \in \Htildesnctaunom{1}{\p}
    \,\Big | \, \vtilden{}_{|\E} \in \Pbb_\p(\E) 
    \quad \forall \E\in\taun \right\}.
\end{equation}
This space does not coincide with~$\Vnim$:
on an edge~$\F$ in $\facets_I$,
the jump of $\vtilden$ in $\Vnimtilde$ is orthogonal 
with respect to polynomials in $\Pbb_{\p-2}(\F)$
and the $\p$-th Legendre polynomial $L_\p^\F$,
instead of being orthogonal
with respect to $\Pbb_{\p-1}(\F)$.

Under the notation in Section~\ref{subsubsection:CR-standard-odd},
we introduce a set of linear functionals similar to those in \eqref{dof-moment:odd}:
on each~$\E$, given~$\vtilde$ in~$W^{1,1}(\E)$,
we consider
\begin{subequations}  \label{dof-moment:new}
\begin{align}
    &|\F|^{-1} (\vtilde, \LjF)_{0,\F}
    & \forall j = 0, \dots\p-2,\p,\;  \quad\forall \F\in \FcalE; \label{dof-moment-edge:new}\\
    & |\E|^{-1} (\vtilde, \mboldbetaE)_{0,\E}
    &     \forall \vert \boldbeta\vert = 0, \dots, \p-3 \label{dof-moment-bulk:new}.
\end{align}
\end{subequations}

\begin{lemma}
The linear functional \eqref{dof-moment-edge:new}-\eqref{dof-moment-bulk:new} are a set of unisolvent DoFs for~$\Pbb_{\p}(\E)$.
\end{lemma}
\begin{proof}
The proof follows along the same lines as that
of Lemma~\ref{lemma:unisolv-odd}.
Let $\vtilden$ be in $\Pbb_p(\E)$ such that the moments \eqref{dof-moment-edge:new}-\eqref{dof-moment-bulk:new} are zero.
The number of linear functionals in \eqref{dof-moment-edge:new}-\eqref{dof-moment-bulk:new}
is equal to $\dim ( \Pbb_p(\E))$;
therefore, it suffices to prove that $\vtilden$ is identically zero over $\E$.

Given a edge~$\F$ in $\FcalE$,
using that the edge moments \eqref{dof-moment-edge:new} vanish,
we deduce that ${\vtilden}{}_{|\F}$ belongs to $\Pbb_p(\F)$ and 
is orthogonal to $\Span\{\Pbb_{\p-2}(\F), L_\p^\F \}$.
Hence, there exists a constant $c$ in $\Rbb$ such that
$\vtilden{}_{| \F}(\xi) = c \, L_{p-1}^\F(\xi)$
for all $\F$ in $\FcalE$.
Since $p$ is even, $L_{p-1}^\F(\xi)$ is odd symmetric with
respect to the midpoint of~$\F$,
and $\vtilden{}_{|\partial \E}$ belongs to $\mathcal{C}^0(\partial \E)$,
then the following identities are valid:
\begin{equation} \label{crucial-step-even-order}
   \vtilden(\nu_1) = -\vtilden(\nu_2) = \vtilden(\nu_3) = -\vtilden(\nu_1),
\end{equation}
whence $\vtilden$ is zero over~$\partial\E$.
Using that the elemental moments are zero,
the assertion follows as in the proof of Lemma~\ref{lemma:unisolv-odd}.
\end{proof}

Following the construction of~$\Vnex$
for odd~$\p$ in~\eqref{explicit-CR:odd},
under the notation in Section~\ref{subsubsection:CR-standard-odd}, we define
\begin{equation}\label{new:explicit-CR}
    \Vnextilde := \Span\left(                
                \bpmoF \text{ and }  \{\phijF\}_{j=1}^{\p-1}\quad
                \forall \F \subset \Fcaln ;\quad
                \{\phiellE\}_{\ell=1}^{(\p-1)(\p-2)/2} 
                \quad \forall \E \in \taun                 \right).
\end{equation}

\begin{lemma} \label{lemma:Vextilde=Vimtilde}
The spaces~$\Vnimtilde$ and~$\Vnextilde$ defined
in~\eqref{new:implicit-CR} and~\eqref{new:explicit-CR}
coincide for even~$p$.
\end{lemma}
\begin{proof}
The proof is exactly the same as that
of Lemma~\ref{lemma:equivalent-definition-spaces:odd}.
\end{proof}

Since the spaces~$\Vnimtilde$ and~$\Vnextilde$ coincide,
we shall write~$\Vtilden$ to denote either of the two;
with an abuse of notation, for the odd degree case,
the spaces~$\Vnim$ and~$\Vnex$
will be henceforth denoted
also by~$\Vnimtilde$ and~$\Vnextilde$.
Given~$g$ in $L^1(\partial \Omega)$, we define
$\Vtildeng$ as $\Htildesncgtaunom{1}{\p} \cap \Vtilden$, where
\footnotesize
\begin{equation} \label{nc-Sobolev-tilde}
\Htildesncgtaunom{1}{\p}
:= \Big\{ \vtilden \in \Htildesnctaunom{1}{\p} \Big |
            (\jump{\vtilden},L^\F_j)_{0,\F} = ( g, L^\F_j)_{0,\F} \quad
            \forall j = 0,\dots,\p-2,\p, \;
                \forall \F\in\FcalnB  \Big\}.
\end{equation}\normalsize

\subsection{Optimal convergence for a stabilized CR method}
\label{subsection:optimal-convergence}
If we replace~$\Vn$ in~\eqref{method-standard}
by~$\Vtilden$ in the case~$\p$ even,
then one may expect method~\eqref{method-standard} to be one order suboptimal.
In fact, proceeding as in the proof of
Lemma~\ref{proposition:standard-convergence},
we rewrite display~\eqref{standard-convergence:1}
as follows:
\[
\begin{split}
\anparent{u}{\vtilden} -(f,\vtilden) 
&= \sum_{\E \in \taun} (\normal{\E} \cdot \nabla u,\vtilden)_{0,\partial \E}
= \sum_{\F \in \Fcaln} (\normal{\F} \cdot  \nabla u , \jump{\vtilden})_{0,\F} \\
&= \sum_{\F \in \Fcaln} \left(\normal{\F} \cdot  \nabla u 
    - \PiltwoF{p-2}( \normal{\F} \cdot  \nabla u ),
    \jump{\vtilden} - \PiltwoF{0} \jump{\vtilden} \right)_{0, \F},
\end{split}                                
\]
which delivers an order of convergence
$\mathcal O(\h^{\p-1})$
due to the presence of $\PiltwoF{\p-2}$
instead of~$\PiltwoF{\p-1}$.

As a matter of notation,
we define the operator~$\PiLpmoF$
as zero if~$\p$ is odd
and as the $L^2(\F)$ projector
onto the Legendre polynomial~$\Lpmo^\F$ if~$\p$ is even.
In order to recover error estimates of order $\mathcal O(\h^{\p})$,
we modify method~\eqref{method-standard}
by employing the discrete bilinear form $\atilden: [\Hone + \Vtilden]^2 \to \Rbb$ defined as follows:
for a sufficiently large positive coefficient~$\eta$,
see Proposition~\ref{proposition:convergence-stabilized-method} below,
\footnotesize
\begin{equation} \label{atilde}
\begin{split}
\antildeparent{\utilden}{\vtilden} 
&:=(\gradh \utilden, \gradh \vtilden)_{0,\Omega}
    - \sum_{\F \in \Fcaln} 
    \left(\PiLpmoF \nbfF \cdot \average{\gradh \utilden}, \PiLpmoF \jump{\vtilden}\right)_{0,\F}  \\
&  - \sum_{\F \in \Fcaln} \left( \PiLpmoF \jump{\utilden}, \PiLpmoF \nbfF \cdot\average{\gradh \vtilden}\right)_{0,\F}
    + \eta \sum_{\F \in \Fcaln } \hF^{-1} 
        \left(\PiLpmoF \jump{\utilden}, \PiLpmoF \jump{\vtilden}\right)_{0,\F}.
\end{split}
\end{equation}
\normalsize
Introduce the mesh-dependent norm
\begin{equation} \label{discrete-norm}
\Normthreebars{\vtilden}^2
:= \SemiNorm{\vtilden}^2_{1,\h;\Omega} 
+ \sum_{\F \in \Fcaln} \hF^{-1} 
        \Norm{\PiLpmoF\jump{\vtilden}}^2_{0,\F}.
\end{equation}
We consider and analyze the following method:
\begin{equation} \label{method-stabilized}
\begin{cases}
    \text{find } \utilden \in \Vtildenzero \text{ such that}\\ 
    \antildeparent{\utilden}{\vtilden}
    =    ( f,\vtilden)_{0,\Omega}
    \qquad\qquad  \forall \vtilden \in \Vtildenzero.
    \end{cases}
\end{equation}
For even~$\p$, the bilinear form~$\atilden$ in~$\eqref{atilde}$
resembles that in standard symmetric interior penalty DG methods~\cite{DiPietro-Ern:2012},
yet requires a stabilization
on a one dimensional space for each edge;
for odd~$\p$, the bilinear form~$\atilden$ instead
coincides with that in~\eqref{broken-bilienarform}.
Comments on an alternative construction
for even degree CR methods
are given in Remark~\ref{remark:alternative-even-fixed} below;
in particular the edge nonconforming bubbles
used in~\eqref{new:explicit-CR}
and the corresponding projections in
\eqref{atilde}
can be of any odd (also linear!) degree
as the crucial point in the unisolvence proof
is the validity of~\eqref{crucial-step-even-order}.

The following well-posedness and convergence result for~\eqref{method-stabilized} holds true.
\begin{proposition} \label{proposition:convergence-stabilized-method}
If~$\eta$ in~\eqref{atilde} is sufficiently large,
then method~\ref{method-stabilized} is well-posed.
Given~$u$ and~$\utilden$ be the solutions to~\eqref{weak:form} and~\eqref{method-stabilized},
we have
\begin{equation}\label{strang:on-tilde}
\Normthreebars{u-\utilden}
\lesssim
\inf_{\vtilden \in \Vtildenzero} \Normthreebars{u-\vtilden} 
+ \sup_{\vtilden \in \Vtildenzero} 
    \frac{|\antildeparent{u}{\vtilden}
    -( f,\vtilden)_{0,\Omega}|}{\Normthreebars{\vtilden}}.
\end{equation}
If~$u$ belongs to $H^{s} (\Omega)$,
$s$ larger than~$3/2$,
and~$\p$ smaller than or equal to~$s$,
we have
\begin{equation} \label{convergence:stabilized-method}
\Normthreebars{u-\utilden}
\lesssim \h^\p \SemiNorm{u}_{\p+1,\Omega}. 
\end{equation}
\end{proposition}
\begin{proof}
The bilinear form in~\eqref{atilde}
is coercive with respect to the mesh-dependent norm~\eqref{discrete-norm}.
Indeed, given~$\vtilden$ in~$\Vtilden$,
definition~\eqref{atilde} yields
\begin{equation} \label{atilde:coercivity}
\atilden(\vtilden,\vtilden) = \sum_{\E \in \taun} \SemiNorm{\vtilden}^2_{1,\E}
- 2 \sum_{\F \in \Fcaln} (\PiLpmoF \nbfF \cdot \average{ \nabla \vtilden }, \PiLpmoF \jump{ \vtilden } )_{0,\F}
+ \eta \sum_{ \F \in \Fcaln} \hF^{-1} \Norm{\PiLpmoF \vtilden }^2_{0,\F}.
\end{equation}
For a given~$\E$ in~$\taun$,
the following discrete trace inequality
\cite[Lemma 1.46]{DiPietro-Ern:2012} holds true:
\begin{equation} \label{trace-inequality:discrete}
\hE^\frac12  \Norm{\vtilden}_{0, \partial \E }
\lesssim \Norm{\vtilden}_{0,\E}.
\end{equation}
On a fixed~$\F$ in~$\Fcaln$,
the Cauchy--Schwarz inequality,
the property~$\hF \le \hE$ for~$\E$ in~$\omegaF$,
the stability of the $L^2$ projector,
and the discrete trace inequality \eqref{trace-inequality:discrete} yield
\small\begin{equation*}
\begin{split}
&\sum_{ \F \in \Fcaln} ( \PiLpmoF \nbfF \cdot \average{\nabla \vtilden},\PiLpmoF \jump{\vtilden})_{0, \F}
\le \big(\sum_{ \F \in \Fcaln} \hF \Norm{ \PiLpmoF \nbfF \cdot \average{\gradh \vtilden}{} }_{0,\F}^2 \big)^\frac12
    \big( \sum_{\F \in \Fcaln} \hF^{-1} \Norm{\PiLpmoF \jump{\vtilden}}^2_{0,\F} \big)^\frac12 \\
& \le \big( \sum_{\E \in \taun}
    \hE \Norm{ \gradh \vtilden{}_{|\E}}^2_{0, \partial \E} \big)^\frac12 
    \big( \sum_{\F \in \Fcaln} \hF^{-1} \Norm{ \PiLpmoF \jump{\vtilden}}^2_{0,\F} \big)^\frac12 
\lesssim \Norm{\gradh \vtilden }_{0,\Omega} \big( \sum_{\F \in \Fcaln} \hF^{-1} \Norm{ \PiLpmoF \jump{\vtilden}}^2_{0,\F} \big)^\frac12.
\end{split}
\end{equation*} \normalsize
Substituting the above inequality in~\eqref{atilde:coercivity}, we obtain
\begin{equation*} 
\atilden(\vtilden,\vtilden)
\gtrsim \Norm{\gradh \vtilden }^2_{0,\Omega} 
    - \Norm{\gradh \vtilden }_{0, \Omega}
    \big(\sum_{\F \in \Fcaln} \hF^{-1} \Norm{ \PiLpmoF \jump{\vtilden}}^2_{0,\F}\big)^{\frac12}
    + \eta \sum_{\F \in \Fcaln} \hF^{-1} \Norm{ \PiLpmoF \jump{\vtilden}}^2_{0,\F}.
\end{equation*}
\normalsize
Further using Young's inequality,
this yields the coercivity of~$\atilden(\cdot,\cdot)$
with respect to the norm~\eqref{discrete-norm}
for~$\eta$ sufficiently large.
On the one hand this implies the well-posedness of
method~\eqref{method-stabilized}.
On the other hand, inequality~\eqref{strang:on-tilde}
follows from the second Strang's lemma;
see, e.g., \cite{Strang:1972,Brenner:2015}.
The hidden constants therein depends on the discrete coercivity constant, which in turns depend on~$\eta$ and~$\p$.
\medskip

In order to derive the error estimate~\eqref{convergence:stabilized-method},
we show an upper bound for the two terms
on the right-hand side of~\eqref{strang:on-tilde} separately.
As for the first one, we observe that
\begin{equation*}
\begin{split}
& \inf_{\vtilden \in \Vtildenzero}
(\sum_{\E \in \taun} \Norm{\nabla(u- \vtilden)}^2_{0,\E} 
+ \sum_{\F \in \Fcaln} \hF^{-1} 
        \Norm{\PiLpmoF(\jump{\vtilden})}^2_{0,\F}) \\
& \le 
\inf_{\vtilden \in \Vtildenzero} (\sum_{\E \in \taun} \Norm{\nabla(u - \vtilden)}^2_{0,\E} 
+ \sum_{\F \in \Fcaln} \hF^{-1} 
        \Norm{\jump{\vtilden}}^2_{0,\F}).  
\end{split}
\end{equation*}
The best error in CR spaces is smaller than
the error in $H^1$ conforming spaces;
in particular the jump terms vanish.
Optimal convergence follows from
polynomial approximation results~\cite{Brenner-Scott:2008}.

As for the second term in~\eqref{strang:on-tilde}, 
convergence for odd~$\p$ is already discussed in
Proposition~\ref{proposition:standard-convergence}.
Therefore, we focus on the even~$\p$ case.
On all edges,
$\average{\nabla u}$, $\jump{u}$,
and~$\PiLpmoF\jump{\vtilden}$
are equal to~$\nabla u$, $0$, and $\jump{\vtilden}$.
These identities, an integration by parts,
problem~\eqref{strong-pois},
and the consistency of the bilinear form~\eqref{atilde} yield
\footnotesize\begin{equation*}
\begin{split}
& \antildeparent{u}{\vtilden} -(f,\vtilden)_{0,\Omega}
 = \sum_{\F \in \Fcaln}
    \left(\nbfF \cdot \nabla u,
    \jump{\vtilden} -\PiltwoF{\p-2} \jump{\vtilden} \right)_{0,\F}  
    - \sum_{\F \in \Fcaln}
    \left(\PiLpmoF (\normal{\F} \cdot \nabla u ),
    \PiLpmoF \jump{\vtilden} \right)_{0,\F} \\
& = \sum_{\F \in \Fcaln}
    \left(\nbfF \cdot \nabla u,
    \jump{\vtilden} -\PiltwoF{\p-1} \jump{\vtilden} \right)_{0,\F} 
 = \sum_{\F \in \Fcaln} \left(\nbfF \cdot \nabla u - \PiltwoF{\p-1} (\normal{\F} \cdot \nabla u ), \jump{\vtilden}-\PiltwoF{0}\jump{\vtilden}\right)_{0,\F}.
\end{split}            
\end{equation*}\normalsize
The assertion now follows along the same lines
as in the proof of Proposition~\ref{proposition:standard-convergence}.
\end{proof}

\begin{remark} \label{remark:simplified-bf-CR}
We may drop the projection operators appearing
in the definition of the the bilinear form in~\eqref{atilde}
and norm in~\eqref{discrete-norm}
since the jumps of CR functions already belong
to the span of a single Legendre polynomial.
In fact, this is what we did in practice in the implementation
of the method; see Section~\ref{section:numerics} below.
An analogous  comment applies to the variable order
scheme detailed in Section~\ref{section:variable-degree} below.
\end{remark}

\section{Numerical aspects for uniform degree CR elements in 2D} \label{section:numerics}

Here, we are interested in assessing the numerical performance of method~\eqref{method-stabilized}.
We focus on the $\h$- and $\p$-versions
of the method in Sections~\ref{subsection:conv-h-version}
and~\ref{subsection:conv-p-version};
since the $\h\p$-version requires additional
technicalities, it will be coped with
in Section~\ref{section:variable-degree} below.
Section~\ref{subsection:stab-parameter}
is concerned with investigating
the effect of the stabilization parameter on the method.
Comparisons with the standard symmetric interior penalty DG
and the standard CR methods are given throughout.

\paragraph*{Test cases.}
We consider $\Omega:= (0,1)^2$
and the exact solutions
\begin{equation} \label{u1}
u_1(x,y) = x+y
\end{equation}
and 
\begin{equation} \label{u2}
u_2(x,y)=\sin(\pi\ x) \sin(\pi\ y).
\end{equation}
The source term and the boundary data
are computed correspondingly.

\paragraph*{Meshes.}
We consider sequences of shape-regular,
quasi-uniform, unstructured simplicial meshes
with halving diameter;
in practice, these meshes are constructed
using the PDE toolbox of Matlab.

\paragraph*{Error measures.}
In what follows, $\un$, $\utilden$, and~$\uDG$
denote the solutions to~\eqref{method-standard},
\eqref{method-stabilized},
and a standard symmetric interior penalty DG method~\cite{DiPietro-Ern:2012}.
The practical realization of the DG method
is based on the use of piecewise modal basis functions,
for a fairer comparison with the other two methods;
we recall for completeness the DG bilinear form
\[
\begin{split}
a_{\text{DG}}(\uDG,\vDG)
:= & (\gradh \uDG, \gradh \vDG)_{0,\Omega}
    - \sum_{\F \in \Fcaln} 
    \left( \nbfF \cdot\average{\gradh \uDG},  \jump{\vDG}\right)_{0,\F}  \\
&  - \sum_{\F \in \Fcaln} \left(  \jump{\uDG}, \nbfF \cdot \average{\gradh \vDG}\right)_{0,\F}
    + \etaDG \sum_{\F \in \Fcaln } \hF^{-1} 
        \left( \jump{\uDG},  \jump{\vDG}\right)_{0,\F},
\end{split}
\]
where~$\etaDG$ is a sufficiently large, positive constant,
which has to scale as~$\p^2$ for stability reasons.

We shall compute the following error measures:
\begin{equation} \label{error-measures}
\ECRp := \frac{\NormthreebarsDG{u-\un}}{\SemiNorm{u}_{1,\Omega}}  ,
\qquad\qquad
\EtildeCRp := \frac{\NormthreebarsDG{u-\utilden}}{\SemiNorm{u}_{1,\Omega}} ,
\qquad\qquad
\EDGp := \frac{\NormthreebarsDG{u-\uDG}}{\SemiNorm{u}_{1,\Omega}},
\end{equation}
where the DG norm is given by
\[
\NormthreebarsDG{\vDG}^2
:= \SemiNorm{\vDG}^2_{1,\h;\Omega} 
  + \sum_{\F \in \Fcaln} \hF^{-1} 
        \Norm{\jump{\vDG}}^2_{0,\F}.
\]
The DG norm for the new CR space
coincides with the norm $\Normthreebars{\cdot}$
in~\eqref{broken:sn-n};
the DG norm for the standard CR space
is stronger than then standard
norm $\SemiNorm{\cdot}_{1,\h;\Omega}$
(albeit the jump terms only involve the projection
onto the scaled Legendre polynomials of order~$\p$).
The reasons why we consider the full DG norm
is to have a fairer comparison
with the other error measures;
very similar results are obtained
with the broken gradient norm.

\paragraph*{On the nestedness of the spaces~$\Vtilden$ in~\eqref{new:explicit-CR}.}
While using the new CR spaces as in~\eqref{new:explicit-CR},
we employ ``essentially'' hierarchical basis functions:
the conforming basis functions are modal,
which are hierarchical as reviewed in Section~\ref{subsection:modal};
the nonconforming bubbles are not hierarchical
but are always one per edge,
differently from the standard CR case.
In fact, as detailed in Remark~\ref{remark:alternative-even-fixed} below,
one could actually employ ``fully''
nested spaces based on using
linear nonconforming bubbles.
As such, it is possible to optimize the matrix assembly process.

\paragraph*{Imposing Dirichlet boundary conditions}

Introduce
\[
\widetilde\Pbb_{\p}^B(\FcalnB)
:= 
\begin{cases}
\Pbb_{\p-1}^B(\FcalnB)                            & \text{if $\p$ is odd}  \\
\Span(\Pbb_{\p-2}(\F),L_\p^\F\mid \F\in\FcalnB) & \text{if $\p$ is even}.
\end{cases}
\]
Possibly inhomogeneous Dirichlet boundary
conditions~$g$ in $H^{\frac12}(\partial\Omega)$
are imposed as in~\eqref{nc-Sobolev-tilde}.
One may alternative solve the (symmetric) mixed problem
\begin{equation}\label{method:with-BCs}
\begin{cases}
\text{find } \utilden \in \Vtilden, \, \btilden \in \Bcaltilden:= \widetilde\Pbb_{\p}^B(\FcalnB) \text{ such that} \\
\anparent{\utilden}{\vtilden} + (\btilden,\vtilden)_{0,\partial \Omega}
= ( f,\vtilden)_{0,\Omega} 
    - (g, \nablah \vtilden \cdot \nbfE)_{0,\partial\Omega}\\
\qquad\qquad\qquad\qquad\qquad\qquad
    + \eta \sum_{\F\in\FcalnB} \hF^{-1} (g,\vtilden)_{0,\partial\E}
                & \forall \vtilden \in \Vtilden \\
(\utilden,\ctilden)_{0,\partial \Omega} = (g, \ctilden)_{0,\partial\Omega}
                &  \forall \ctilden \in \Bcaltilden .
\end{cases}
\end{equation}

\subsection{Numerical results: \texorpdfstring{$\h$}{h}-version}
\label{subsection:conv-h-version}
Here, we investigate the convergence
of the errors in~\eqref{error-measures}
under $\h$-refinements.
Since method~\eqref{method-stabilized}
coincides with the standard CR method~\eqref{method-standard} for odd~$\p$,
we consider even integers~$\p$,
more precisely $\p=2$ and~$\p=4$.
We employ $\eta=\etaDG=5$ and~$20$
for the two polynomial degrees.

First, we consider the exact solution~$u_1$
in~\eqref{u1}, which belongs to the CR and DG spaces above
and is thus reproduced by all methods up to round-off errors;
we display the results in Figure~\ref{fig:condition-number}.
This test case shows how the propagation
of the floating point errors depends on the number
of degrees of freedom.

\begin{figure}[H]
    \centering
\includegraphics[width = 0.45\linewidth]{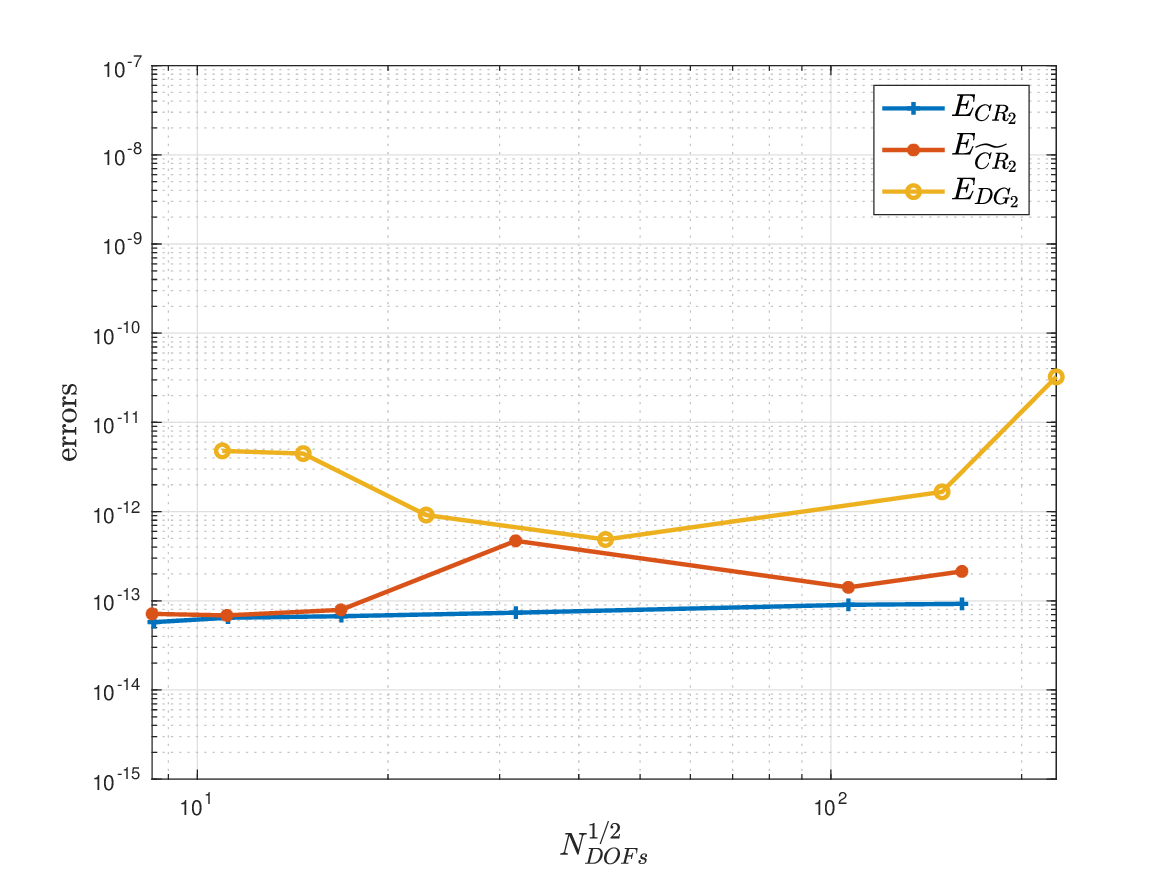}
\includegraphics[width = 0.45\linewidth]{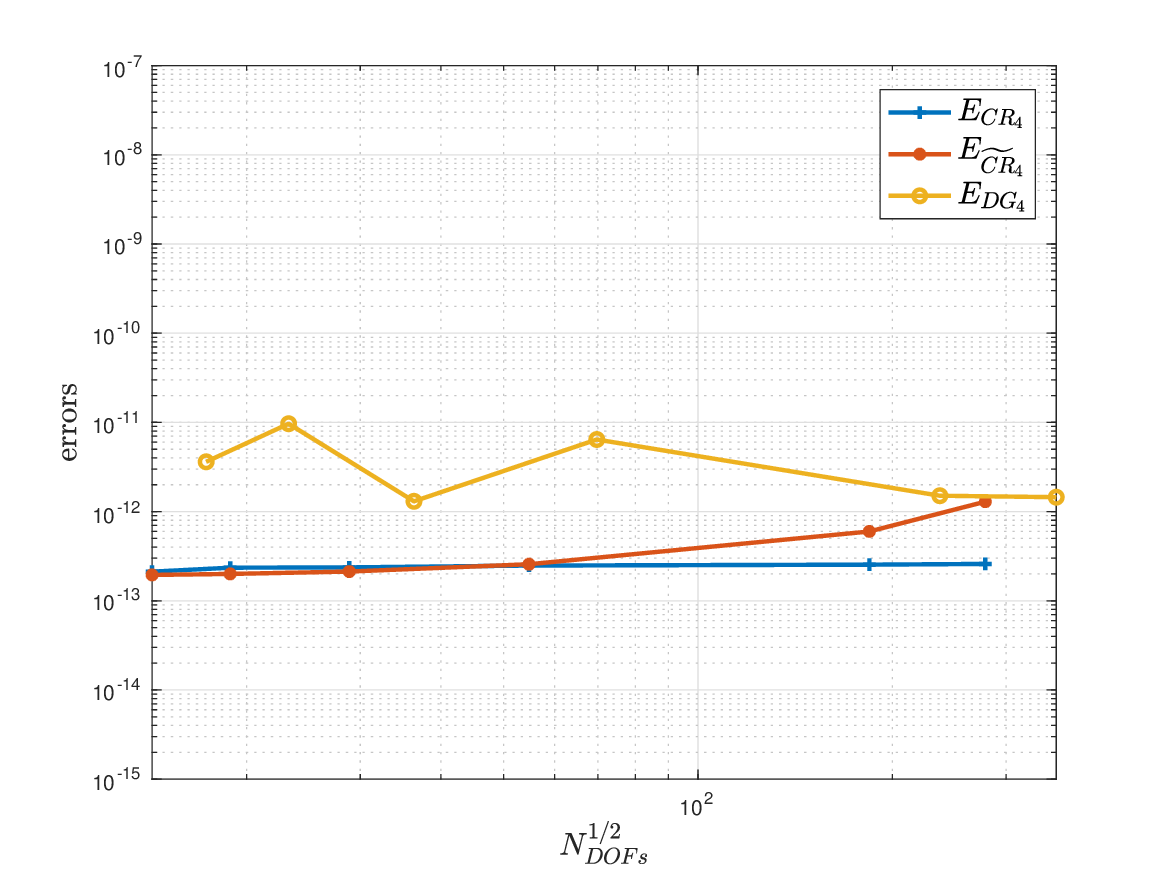}
\caption{Exact solution~$u_1$ in~\eqref{u1};
$\p=2$ and~$\eta=\etaDG=5$ (left panel);
$\p=4$ and~$\eta=\etaDG=20$ (right panel).
We compare the error measures in~\eqref{error-measures}
under $\h$-refinements.} \label{fig:condition-number}
\end{figure}

The errors of the two CR methods
are smaller than that of the DG one.
This is probably due to the fact that in the standard
CR method no stabilization is employed
whereas in the new one
the stabilization is smaller than that for the DG one.
\medskip

Next, we consider the exact solution~$u_2$
in~\eqref{u2},
and $\eta=20$
for the polynomial degrees~$2$ and~$4$,
and display the results in 
Figure~\ref{fig:h-convergence}.
\begin{figure}[H]
\centering
\includegraphics[width=0.45\linewidth]{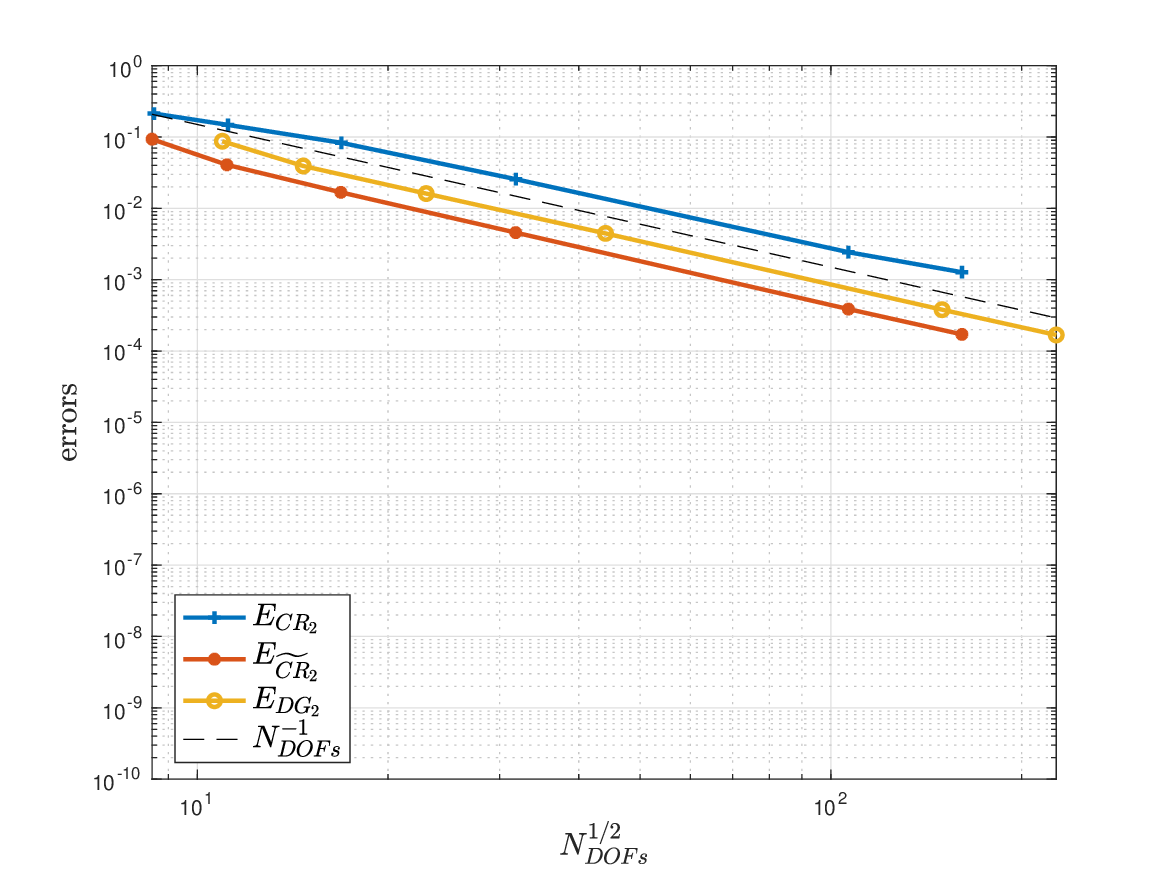}  \includegraphics[width=0.45\linewidth]{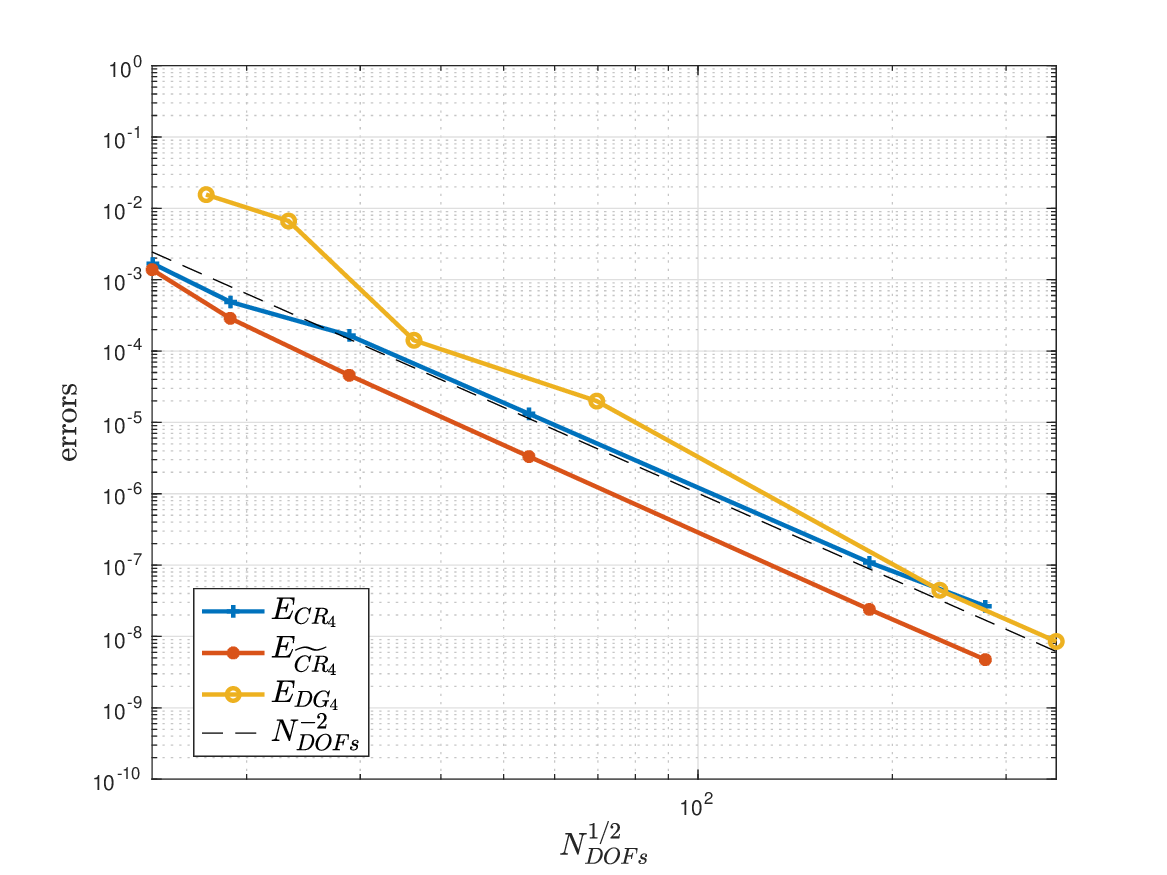}
\caption{Exact solution~$u_2$ in~\eqref{u2};
$\eta=\etaDG=20$;
$\p=2$ (left panel);
$\p=4$ (right panel).
We compare the error measures in~\eqref{error-measures}
under $\h$-refinements.} \label{fig:h-convergence}
\end{figure}

In all cases, we observe rates of convergence
that are in accordance with the theoretical findings.

\subsection{Numerical results: \texorpdfstring{$\p$}{p}-version}
\label{subsection:conv-p-version}

Here, we investigate the convergence
of the errors in~\eqref{error-measures}
under $\p$-refinements
using a fixed coarse simplicial mesh of $88$ elements.
We employ
$\eta=\etaDG= 5 \p^2$,
consider the exact solution~$u_2$ in~\eqref{u2},
and display the results in 
Figure~\ref{fig:p-convergence-variable-eta}.

\begin{figure}[H]
\centering
\includegraphics[width=0.45\linewidth]{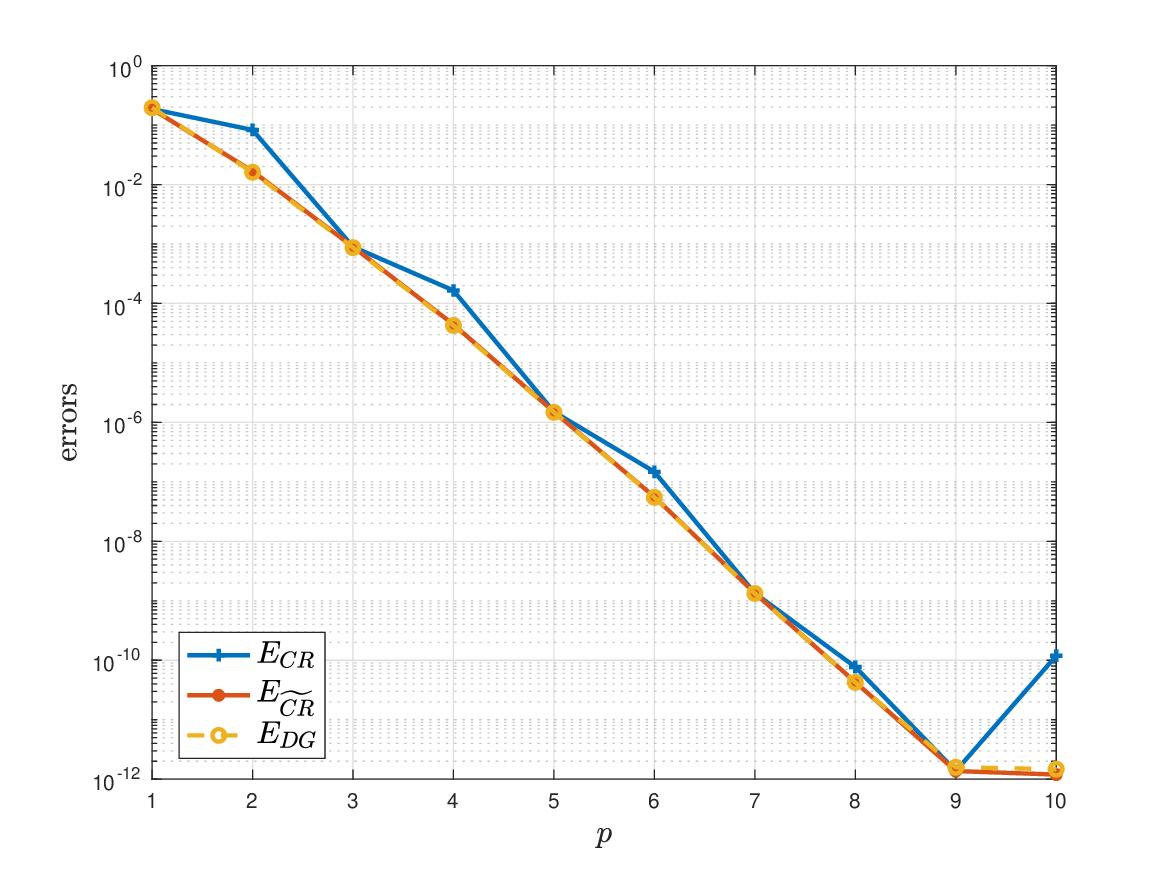} 
\includegraphics[width=0.45\linewidth]{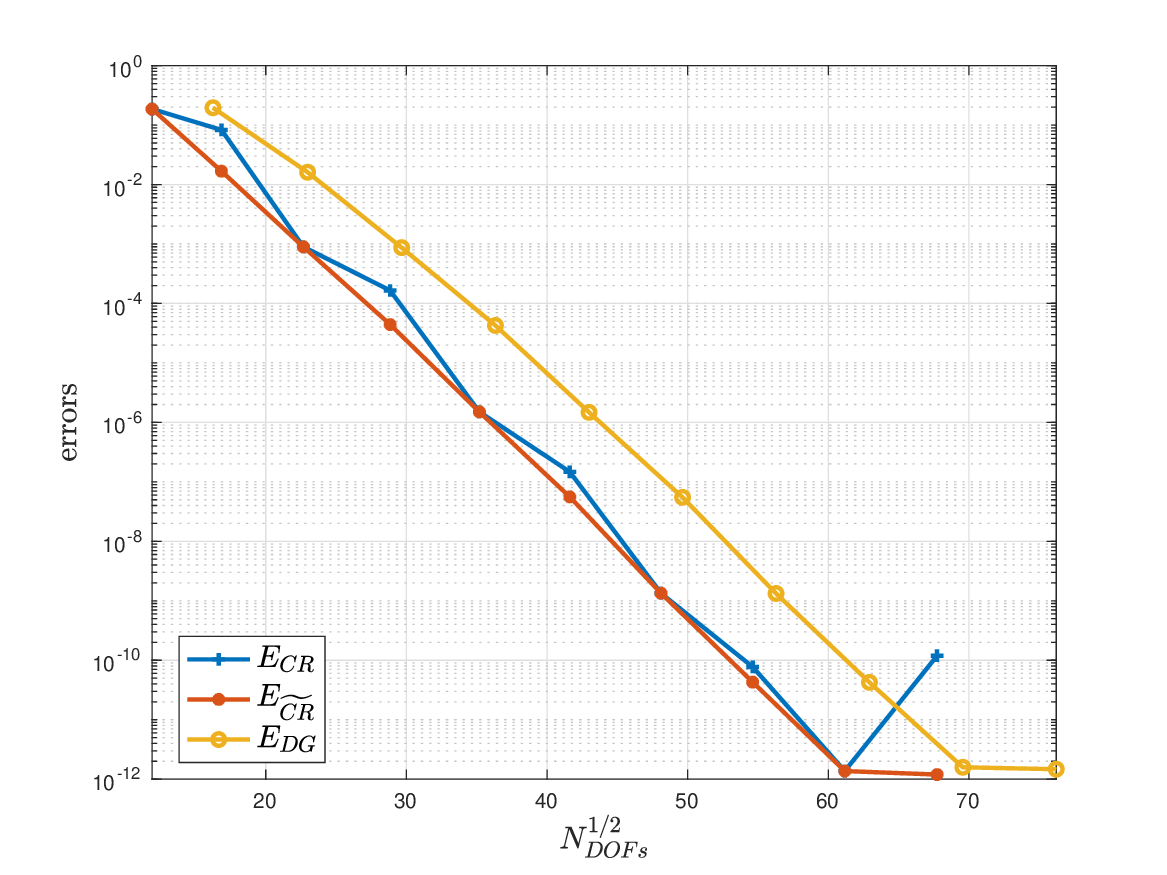}
\caption{Exact solution~$u_2$ in~\eqref{u2};
coarse fixed simplicial mesh;
$\eta=\etaDG= 5 \p^2$.
We compare the error measures in~\eqref{error-measures}
under $\p$-refinements with respect to the
polynomial degree (left panel)
and the square root of the number of degrees of freedom (right panel).}
\label{fig:p-convergence-variable-eta}
\end{figure}

As expected from standard $\p$-approximation theory,
see Remark~\ref{remark:hp-version} below,
we observe exponential convergence in terms of~$\p$.

\subsection{On the choice of the stabilization parameter}
\label{subsection:stab-parameter}

We investigate here the performance
of the errors $\EtildeCRp$ and~$\EDGp$
taking different values of $\eta$ and~$\etaDG$.
We consider the exact solution~$u_2$ in~\eqref{u2}
and compare the errors in~\eqref{error-measures}
in terms of different choices of the stabilization parameters
$\eta$ and~$\etaDG$,
say, in the set $\{0.5; 1; 2; 4; 8; 16; 32; 64\}$.
We pick $\p$ in $\{2 , 4,  6, 8 \}$.
We display the results in Figure~\ref{fig:choice-eta}

\begin{figure}[H]
\centering
\includegraphics[width=0.45\linewidth]{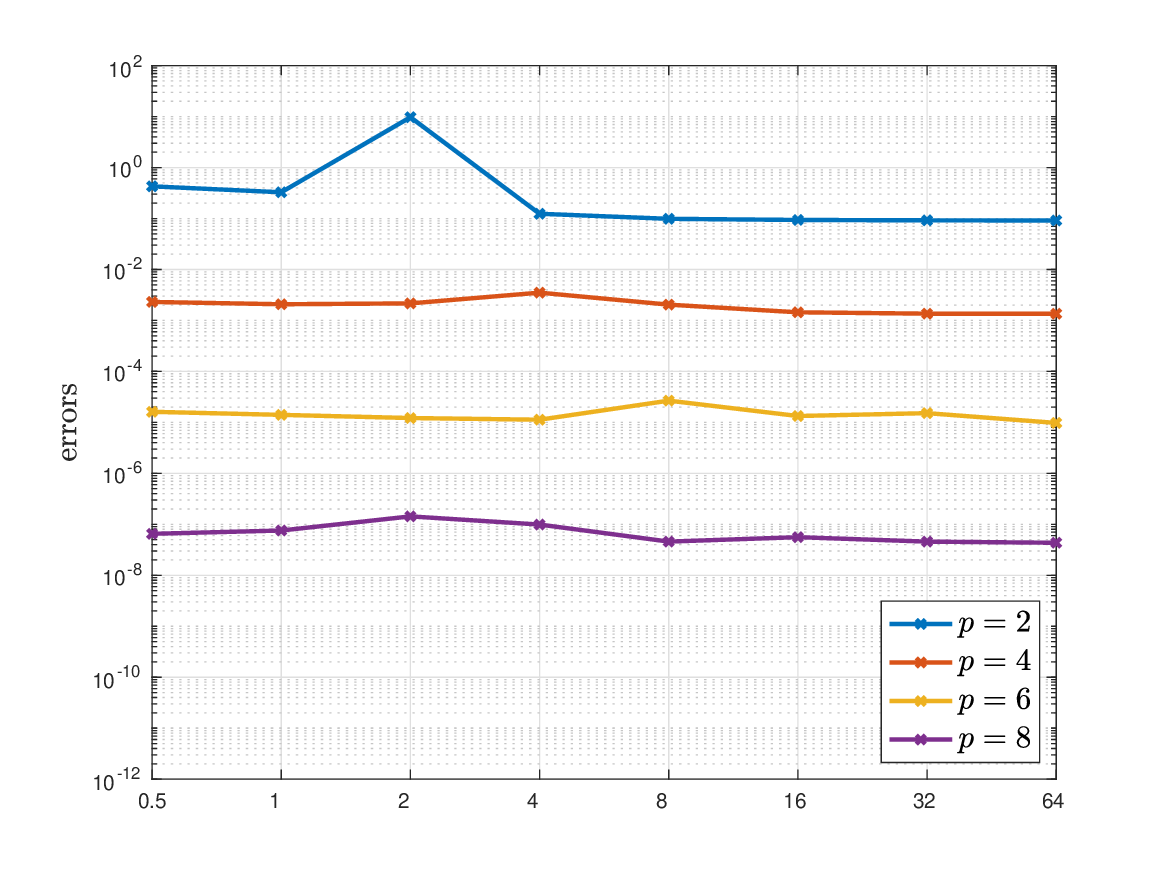}
\includegraphics[width=0.45\linewidth]{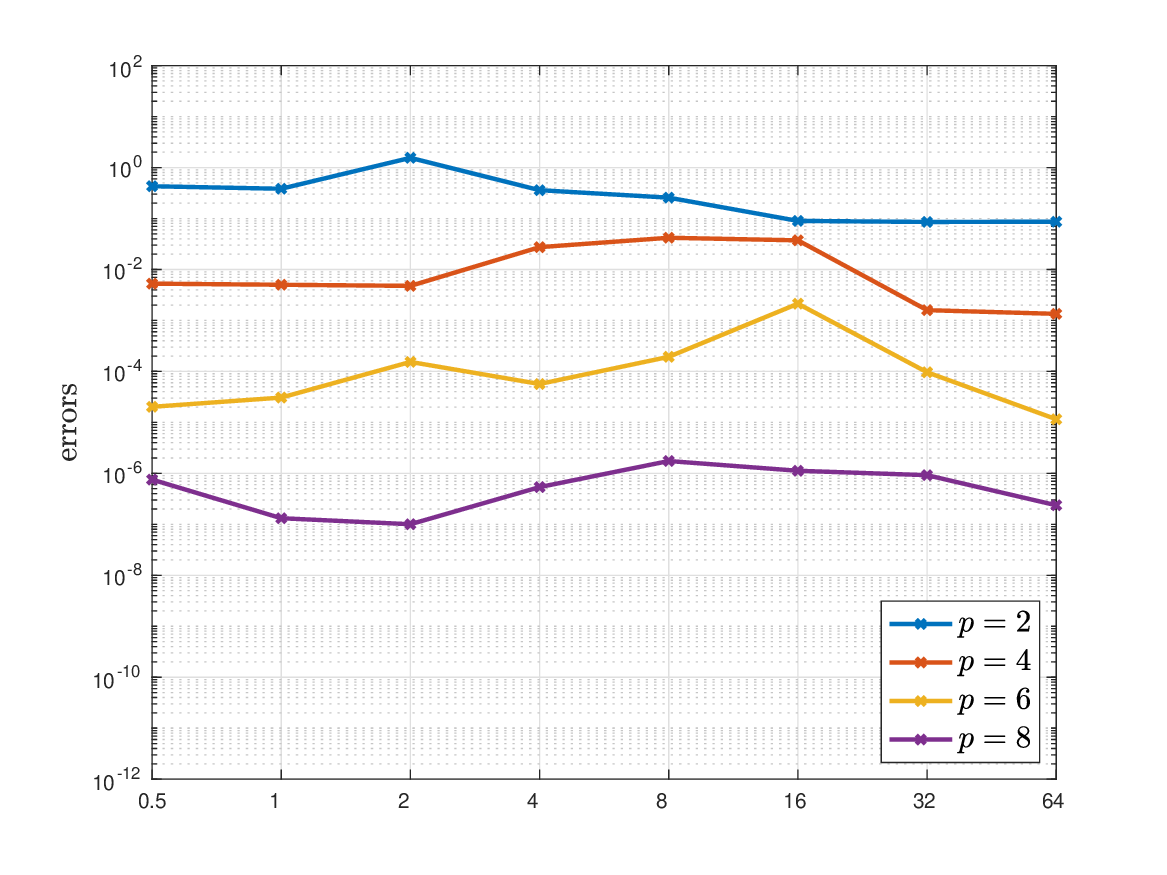}
\caption{Exact solution $u_2$ in~\eqref{u2};
$\p=2$, $4$, $6$, and $8$; fixed coarse simplicial mesh.
We compare $\EtildeCRp$ (left panel)
and $\EDGp$ (right panel)
picking $\eta$ in
$\{0.5; 1; 2; 4; 8; 16; 32; 64\}$.
}
\label{fig:choice-eta}
\end{figure}

It is apparent that, especially for high degree $\p$,
the errors for the DG method display a more oscillatory behaviour
compared to those for the modified CR method.

\section{Variable degree Crouzeix-Raviart elements} \label{section:variable-degree}
The aim of this section is to construct
and analyze variable degree CR methods.
After discussing why an ``obvious'' construction
cannot work in Section~\ref{subsection:variable-obvious-cannot-work},
in Section~\ref{subsection:variable-global-spaces}
we propose new variable order CR spaces;
Section~\ref{subsection:variable-method} is devoted to
the design and analysis of a corresponding variable order method;
in Section~\ref{subsection:nr-p-hp}, we assess
the performance of that method
with a particular emphasis on the approximation
of corner singularities under $\h\p$-refinements.

\paragraph*{Notation for variable order CR spaces.}
For any mesh~$\mesh$ in a given sequence,
we can assume that its elements are in bijection with
$\{1,2,\dots,\card(\taun)\}$;
let~$\pbf$ be in~$\mathbb N^{\card(\taun)}$.
With the element~$\E$ (in bijection with the index~$j$),
we associate the polynomial degree~$\pbf_j$
and denote it also by~$\p_{\E}$ on occasions.
In what follows, with an abuse of notation,
for a generic interior edge~$\F$,
$\Eo$ and~$\Et$ denote the two elements sharing~$\F$;
for a generic boundary edge~$\F$,
$\E$ denotes the element such that $\F = \partial\E\cap \partial\Omega$.

With each edge~$\F$,
we associate
\[
\pF:=
\begin{cases}
\max\{\pEo,\pEt\}   & \text{if } \F = \partial\Eo\cap\partial\Et \\
\pE                 & \text{if } \F = \partial\E\cap\partial\Omega.
\end{cases}
\]

\subsection{An ``obvious'' variable order CR space does not work} \label{subsection:variable-obvious-cannot-work}
We exhibit a counterexample to show why an ``obvious''
construction of variable order CR spaces
(and the corresponding method)
cannot deliver a sequence of discrete solutions
converging to the exact one.

Given~$\p$ in~$\mathbb N$, we write
\[
\ptilde
:= \begin{cases}
    \p   &\text{if $\p$ is odd} \\
    \p-1 &\text{if $\p$ is even}.
\end{cases}
\]
With each internal edge~$\F$,
we associate a double-sided nonconforming bubble~$\bptildeotF$ such that
\begin{equation} \label{two-sided-bubbles}
\bptildeotF{}_{|\Ej}
:=
\begin{cases}
\bptildeoF     & \text{in } \E_1 \\
\bptildetF     & \text{in } \E_2 \\
    0          & \text{elsewhere},
\end{cases}
\end{equation}
where~$\bptildeoF$ and~$\bptildetF$ are the nonconforming bubble functions
as in~\eqref{nc-edge:functions}.
A similar construction applies for boundary edges
(only a one-sided nonconforming bubble~$\bptildeEF$ appears).
As in the uniform polynomial degree case,
we have~$\bptildeotF{}_{|\F}=1$.
Besides, $\bptildeotF$ is nonconforming on the (four) edges of~$\omegaF$ in~\eqref{facet-patch}.

An   ``obvious'' construction is as follows:
\begin{itemize}
    \item on each element~$\E$, consider modal polynomials of order
    $\pE$ excluding the (linear) hat functions;
    \item associate with each internal edge~$\F$ shared
    by the elements~$\Eo$ and~$\Et$,
    the double-sided nonconforming bubble function in~\eqref{two-sided-bubbles};
    \item associate with each boundary edge~$\F$ on the boundary
    of the element~$\E$,
    the one-sided nonconforming bubble function of order~$\pE$.
\end{itemize}
An ``obvious'' CR variable order space is given
by the span of the functions described
in the above bullets.
We wish that a method associated with this space
is at least first order consistent,
i.e., that it reproduces affine solutions exactly
up to machine precision.

Consider a mesh of four elements
and a CR variable order space
corresponding to polynomial degrees $(1,1,1,3)$
as detailed in Figure~\ref{figure:1113}.
\begin{figure}[H]
\centering
\begin{center}
\begin{tikzpicture}[scale=0.9]
\draw[black, very thick, -] (0,0) -- (4,0) -- (4,4) -- (0,4) -- (0,0);
\draw[very thick, -] (0,0) -- (4,4);
\draw[very thick, -] (0,4) -- (4,0);
\draw[red,fill=red] (2,0) circle (.5ex);
\draw[red,fill=red] (4,2) circle (.5ex);
\draw[red,fill=red] (2,4) circle (.5ex);
\draw[red,fill=red] (3,1) circle (.5ex);
\draw[red,fill=red] (3,3) circle (.5ex);
\draw[purple,fill=purple] (1,1) circle (.5ex);
\draw[purple,fill=purple] (1,3) circle (.5ex);
\draw[blue,fill=blue] (0,2) circle (.5ex);
\node[fill=orange,regular polygon, regular polygon sides=4,inner sep=3pt] at (.5,.5) {};
\node[fill=orange,regular polygon, regular polygon sides=4,inner sep=3pt] at (1.5,1.5) {};
\node[fill=orange,regular polygon, regular polygon sides=4,inner sep=3pt] at (.5,3.5) {};
\node[fill=orange,regular polygon, regular polygon sides=4,inner sep=3pt] at (1.5,2.5) {};
\node[fill=orange,regular polygon, regular polygon sides=4,inner sep=3pt] at (0,1) {};
\node[fill=orange,regular polygon, regular polygon sides=4,inner sep=3pt] at (0,3) {};
\node[fill=orange,regular polygon, regular polygon sides=4,inner sep=3pt] at (.7,2) {};
\draw (.9,.2) node[black, left] {$1$};
\draw (3.98,.5) node[black, left] {$1$};
\draw (3.5,3.7) node[black, left] {$1$};
\draw (.375,3.5) node[black, left] {$3$};
\end{tikzpicture}
\end{center}
\caption{Graphical representation of the degrees of freedom
for an ``obvious'' variable order CR space
over a mesh of four elements
and polynomial degree distribution given by $(1,1,1,3)$.
The circles stand for the coefficients of the nonconforming bubbles:
in red and blue, the two-sided and one-sided
bubble functions of order~$1$ and~$3$ respectively;
in purple the two-sided bubble functions
of mixed orders~$1$ and~$3$.
The squares stand for the coefficients of the modal functions.
Inside each element, we write the corresponding local order.}
\label{figure:1113}
\end{figure}
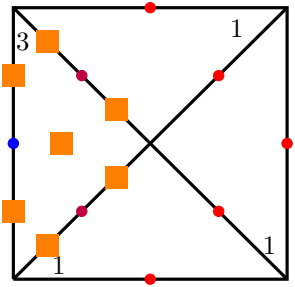

In order to represent, e.g., an affine
function on the left element
that is nonconstant on the three edges,
the coefficients of the nonconforming bubbles
on that element are fixed by the three vertex values;
moreover, as the involved bubbles
are cubic polynomials on that element,
surely all the coefficients of the cubic
modal basis functions cannot vanish.
Bearing this in mind, that linear function,
e.g., on the bottom element,
should be represented by the linear combination
of the linear nonconforming bubbles
(and this alone would be OK)
\emph{and} the modal basis functions
where the cubic one has nonzero coefficient.
This is clearly not possible,
whence a cleverer construction must be in order;
see Section~\ref{subsection:variable-global-spaces} below.
In a sense, this lack of linear consistency
might be regarded as a locking property
of the ``obvious'' construction of variable order
CR spaces as detailed above.

The main philosophy that we shall use
in order to fix this issue is as follows:
\begin{itemize}
    \item in order to avoid the locking,
    we shall always use linear nonconforming bubble functions
    (it can be readily checked on the example in Figure~\ref{figure:1113} that in this case the modal
    basis functions on the edges of the left element
    are not active);
    \item since linear nonconforming bubbles have nonvanishing
    linear component jumps, we shall correct the method
    similarly to what we did in Section~\ref{subsubsection:CR-standard-even}
    for all orders, regardless of its parity.
\end{itemize}

\subsection{Construction of variable degree CR global spaces} \label{subsection:variable-global-spaces}
We are now in a position to exhibit
an explicit version variable order
CR space associated with a mesh~$\taun$
and a polynomial degree distribution~$\pbf$,
which do not display the locking discussed in Section~\ref{subsection:variable-obvious-cannot-work}.
Given modal basis functions as in~\eqref{modal-facet}
and~\eqref{modal-element},
and the linear ($\p=1$) nonconforming bubble
functions as in~\eqref{nc-edge:functions},
we define
\begin{equation} \label{space:variable-degree-explicit}
\Vnextilde := \Span\left(                
                \btildeoF \text{ and }  \{\phijF\}_{j=1}^{\pF-1}\quad
                \forall \F \subset \Fcaln ;\quad
                \{\phiellE\}_{\ell=1}^{(\pE-1)(\pE-2)/2} 
                \quad \forall \E \in \taun                 \right).
\end{equation}

\begin{lemma} \label{lemma:unisolvence-variable-order}
The spanning set~\eqref{space:variable-degree-explicit} consists of linearly independent functions.
\end{lemma}
\begin{proof}
On each element, the span of the local modal basis
functions is a set of linearly independent functions,
which does not contain the space
of continuous piecewise affine functions.
That space is fixed by including also the linear nonconforming bubbles.
\end{proof}

A set of unisolvent degrees of freedom for the explicit space is then given by the coefficients
in the expansion with respect to the basis in~\eqref{space:variable-degree-explicit};
see Figure~\ref{figure:variable-degreee-dofs} for a graphical representation
with different degree distributions~$\pbf$.

\begin{figure}[htbp]
\centering
\begin{tikzpicture}[scale = 2]
\coordinate (A) at (0,0);
\coordinate (B) at (1,0);
\coordinate (C) at (1,1);
\coordinate (D) at (0,1);
\coordinate (E) at (1,2);
\coordinate (F) at (0,2);
\coordinate (a) at (0.25,0.25);
\coordinate (b) at (0.75,0.75);
\coordinate (c) at (0.25,1.25);
\coordinate (d) at (0.75,1.75);
\coordinate (o1) at (0.25,0.25);
\coordinate (o2) at (0.75,0.75);
\coordinate (o3) at (0.25,1.25);
\coordinate (o4) at (0.75,1.75);
\coordinate (o) at (0,0);
\draw (A) -- (B) -- (E) -- (F) -- cycle;
\draw (B) -- (D) -- (C) -- (F);
\node[] at ([xshift=-0.25cm,yshift=0.1cm]B) {$3$};
\node[] at ([xshift=-0.25cm,yshift=0.1cm]C) {$3$};
\node[] at ([xshift=0.25cm,yshift=-0.1cm]D) {$3$};
\node[] at ([xshift=0.25cm,yshift=-0.1cm]F) {$3$};
\node[orange] at ([xshift=0.25cm,yshift=0.25cm]A) {$\sqbullet$};
\node[orange] at ([xshift=0.75cm,yshift=0.75cm]A) {$\sqbullet$};
\node[orange] at ([xshift=0.25cm,yshift=1.25cm]A) {$\sqbullet$};
\node[orange] at ([xshift=0.75cm,yshift=1.75cm]A) {$\sqbullet$};

\foreach[evaluate=\shiftyE using \eley] \eley in {0,1,2}{
\foreach[evaluate=\shiftxE using \elex] \elex in {0,1}{
\foreach[evaluate=\x as \shiftx using 0.125*\x,
         evaluate=\y as \shifty using 0.125*(8-\x)+\shiftyE,
         evaluate=\shiftxb using 0.125*(8-\x),
         evaluate=\shiftyb using 0.125*(\x)+\shiftyE] \x in {1,2,4}{
\ifthenelse{\eley = 2}{\ifthenelse{\x = 4}{\node[green] at ([xshift=\shiftx cm, yshift = \shiftyE cm]o) {$\bullet$};}{\node[orange] at ([xshift=\shiftx cm, yshift = \shiftyE cm]o) {$ \sqbullet$};}}{
\ifthenelse{\x = 4}{\node[green] at ([xshift=\shiftx cm, yshift = \shiftyE cm]o) {$\bullet$};}{\node[orange] at ([xshift=\shiftx cm, yshift = \shiftyE cm]o) {$ \sqbullet$};}
\ifthenelse{\x = 4}{\node[green] at ([xshift = \shiftxE cm, yshift=\shifty cm]o) {$\bullet$};}{\node[orange] at ([xshift = \shiftxE cm, yshift=\shifty cm]o) {$ \sqbullet$};}
\ifthenelse{\x = 4}{\node[green] at ([xshift = \shiftxb cm, yshift=\shiftyb cm]o) {$\bullet$};}{\node[orange] at ([xshift = \shiftxb cm, yshift=\shiftyb cm]o) {$ \sqbullet$};}}
}}}

\draw ([xshift=2cm]A) -- ([xshift=2cm]B) -- ([xshift=2cm]E) -- ([xshift=2cm]F) -- cycle;
\draw ([xshift=2cm]B) -- ([xshift=2cm]D) -- ([xshift=2cm]C) -- ([xshift=2cm]F);
\node[] at ([xshift=1.75cm,yshift=0.1cm]B) {$2$};
\node[] at ([xshift=1.75cm,yshift=0.1cm]C) {$2$};
\node[] at ([xshift=2.25cm,yshift=-0.1cm]D) {$2$};
\node[] at ([xshift=2.25cm,yshift=-0.1cm]F) {$2$};

\foreach[evaluate=\shiftyE using \eley] \eley in {0,1,2}{
\foreach[evaluate=\shiftxE using 2+\elex] \elex in {0,1}{
\foreach[evaluate=\x as \shiftx using 0.125*\x+2,
         evaluate=\y as \shifty using 0.125*(8-\x)+\shiftyE,
         evaluate=\shiftxb using 0.125*(8-\x)+2,
         evaluate=\shiftyb using 0.125*(\x)+\shiftyE] \x in {1,4}{
\ifthenelse{\eley = 2}{\ifthenelse{\x = 4}{\node[green] at ([xshift=\shiftx cm, yshift = \shiftyE cm]o) {$\bullet$};}{\node[orange] at ([xshift=\shiftx cm, yshift = \shiftyE cm]o) {$ \sqbullet$};}}{
\ifthenelse{\x = 4}{\node[green] at ([xshift=\shiftx cm, yshift = \shiftyE cm]o) {$\bullet$};}{\node[orange] at ([xshift=\shiftx cm, yshift = \shiftyE cm]o) {$ \sqbullet$};}
\ifthenelse{\x = 4}{\node[green] at ([xshift = \shiftxE cm, yshift=\shifty cm]o) {$\bullet$};}{\node[orange] at ([xshift = \shiftxE cm, yshift=\shifty cm]o) {$ \sqbullet$};}
\ifthenelse{\x = 4}{\node[green] at ([xshift = \shiftxb cm, yshift=\shiftyb cm]o) {$\bullet$};}{\node[orange] at ([xshift = \shiftxb cm, yshift=\shiftyb cm]o) {$ \sqbullet$};}}
}}}

\draw ([xshift=4cm]A) -- ([xshift=4cm]B) -- ([xshift=4cm]E) -- ([xshift=4cm]F) -- cycle;
\draw ([xshift=4cm]B) -- ([xshift=4cm]D) -- ([xshift=4cm]C) -- ([xshift=4cm]F);
\node[] at ([xshift=3.75cm,yshift=0.1cm]B) {$4$};
\node[] at ([xshift=3.75cm,yshift=0.1cm]C) {$2$};
\node[] at ([xshift=4.25cm,yshift=-0.1cm]D) {$3$};
\node[] at ([xshift=4.25cm,yshift=-0.1cm]F) {$1$};


\node[orange] at ([xshift=4.25cm,yshift=0.25cm]A) {$\sqbullet$};
\node[orange] at ([xshift=4.4cm,yshift=0.25cm]A) {$\sqbullet$};
\node[orange] at ([xshift=4.25cm,yshift=0.4cm]A) {$\sqbullet$};

\foreach[evaluate=\x as \shiftx using 0.125*\x+4,
         evaluate=\y as \shifty using 0.125*(8-\x),
         evaluate=\shiftxb using 0.125*(8-\x)+4,
         evaluate=\shiftyb using 0.125*(\x)] \x in {1,...,4}{
\ifthenelse{\x = 4}{\node[green] at ([xshift=\shiftx cm]o) {$\bullet$};}{\node[orange] at ([xshift=\shiftx cm]o) {$ \sqbullet$};}
\ifthenelse{\x = 4}{\node[green] at ([xshift = 4cm, yshift=\shifty cm]o) {$\bullet$};}{\node[orange] at ([xshift = 4cm, yshift=\shifty cm]o) {$ \sqbullet$};}
\ifthenelse{\x = 4}{\node[green] at ([xshift = \shiftxb cm, yshift=\shiftyb cm]o) {$\bullet$};}{\node[orange] at ([xshift = \shiftxb cm, yshift=\shiftyb cm]o) {$ \sqbullet$};}
}

\node[orange] at ([xshift=4.75cm,yshift=0.75cm]A) {$\sqbullet$};

\foreach[evaluate=\x as \shiftx using 0.125*(8-\x)+4,
         evaluate=\y as \shifty using 0.125*(\x)] \x in {1,2,4}{
\ifthenelse{\x = 4}{\node[green] at ([xshift=\shiftx cm, yshift = 1cm]o) {$\bullet$};}{\node[orange] at ([xshift=\shiftx cm, yshift = 1cm]o) {$ \sqbullet$};}
\ifthenelse{\x = 4}{\node[green] at ([xshift = 5cm, yshift=\shifty cm]o) {$\bullet$};}{\node[orange] at ([xshift = 5cm, yshift=\shifty cm]o) {$ \sqbullet$};}
\ifthenelse{\x = 4}{\node[green] at ([xshift = \shiftx cm, yshift=\shifty cm]o) {$\bullet$};}
}

\foreach[evaluate=\x as \shiftx using 0.125*(\x)+4,
         evaluate=\y as \shifty using 0.125*(8-\x)+1,
         evaluate=\shiftyb using 0.125*(\x)+1,
         evaluate=\shiftxb using 0.125*(8-\x)+4] \x in {1,4}{
\ifthenelse{\x = 4}{\node[green] at ([xshift=\shiftx cm, yshift = 1cm]o) {$\bullet$};}{}
\ifthenelse{\x = 4}{\node[green] at ([xshift = 4cm, yshift=\shifty cm]o) {$\bullet$};}
{\node[orange] at ([xshift = 4cm, yshift=\shifty cm]o) {$ \sqbullet$};}
\ifthenelse{\x = 4}{\node[green] at ([xshift = \shiftxb cm, yshift=\shiftyb cm]o) {$\bullet$};}{\node[orange] at ([xshift = \shiftxb cm, yshift=\shiftyb cm]o) {$ \sqbullet$};}
}

\foreach[evaluate=\x as \shiftx using 0.125*(\x)+4,
         evaluate=\y as \shifty using 0.125*(8-\x)+1] \x in {4}{
\node[green] at ([xshift=\shiftx cm, yshift = 2 cm]o) {$\bullet$};
\node[green] at ([xshift = 5cm, yshift=\shifty cm]o) {$\bullet$};
\node[green] at ([xshift = \shiftx cm, yshift=\shifty cm]o) {$\bullet$};
}
\end{tikzpicture}
\caption{Examples of global CR-type DoFs chosen as modal basis functions ($\orange{\sqbullet}$) and linear nonconforming bubbles ($\green{\bullet}$) according to~\eqref{space:variable-degree-explicit}  with different degree distributions~$\pbf$.}
\label{figure:variable-degreee-dofs}
\end{figure}
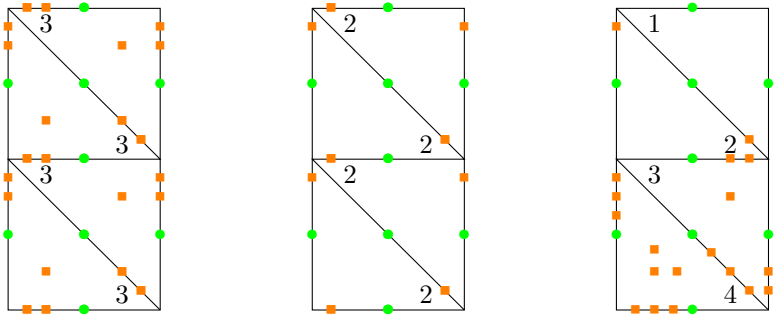

We define
\[ 
\Htildesnctaunom{1}{\pbf} := \left\{ \vn \in H^1(\taun,\Omega) \,\Big | \,
    (\jump{\vn}_F,\LjF)_{0,\F}=0 \quad
    \forall j = 0,\dots,\pF,\, 
    j \neq 1, \, \forall \F \in \FcalnI \right\}.
\]\normalsize
Given~$g$ in $L^1(\partial \Omega)$, we define
\[
\Vtildenexg  = \Htildesncgtaunom{1}{\pbf} \cap \Vnextilde,
\]
where
\footnotesize
\[ 
\Htildesncgtaunom{1}{\pbf}
:= \Big\{ v \in \Htildesnctaunom{1}{\pbf} \Big |
                (\jump{\vn},L^\F_j)_{0,\F} 
                = (g, L^\F_j)_{0,\F} \quad
                \forall j = 0,\dots,\pF, \; j \ne 1, \quad
                \forall \F\in\FcalnB  \Big\}.
\]
\normalsize
In what follows, we shall write~$\Vtildenzero$ instead of~$\Vtildenexg$.

\begin{remark} \label{remark:variable-order-implicit-VS-explicit}
$\Vnextilde$ is contained in
$\Htildesnctaunom{1}{\pbf}$;
the proof of this fact is essentially contained in the first part
of Lemmas~\ref{lemma:equivalent-definition-spaces:odd}
and~\ref{lemma:equivalent-definition-spaces:even}.
\end{remark}

\subsection{A variable degree CR method} \label{subsection:variable-method}
We define the operator~$\PiLo$
as the~$L^2(\F)$ projector onto the scaled Legendre polynomial
of order~$1$ over~$\F$
and the operator~$\PitildeLo$
as the~$L^2(\F)$ projector onto the span of
\[
\{L^\F_j \mid j=0,\dots, \pF-1,\, j \ne 1\}.
\]
We introduce the discrete bilinear form
$\atilden: [\Hone + \Vtilden]^2 \to \Rbb$
defined as follows:
for sufficiently large positive coefficients~$\etaF$
for all edges~$\F$,
see Proposition~\ref{proposition:convergence-variable-order} below,
\small\begin{equation} \label{anptilde}
\begin{split}
\antildeparent{\utilden}{\vtilden} 
& :=(\gradh \utilden, \gradh \vtilden)_{0,\Omega}
  - \sum_{\F \in \Fcaln}
    \left(\PiLo \nbfF \cdot \average{\gradh \utilden}, \PiLo \jump{\vtilden}\right)_{0,\F} \\
&\quad - \sum_{\F \in \Fcaln}
    \left( \PiLo \jump{\utilden}, \PiLo \nbfF \cdot \average{\gradh \vtilden}\right)_{0,\F}     
    +  \sum_{\F \in \Fcaln} \etaF
         \hF^{-1} \left(\PiLo \jump{\utilden}, \PiLo \jump{\vtilden}\right)_{0,\F}.
\end{split}
\end{equation}\normalsize
Introduce the mesh-dependent norm,
for which Remark~\ref{remark:simplified-bf-CR}
still applies,
\begin{equation} \label{variable-degree:norm}
\Normthreebarshp{\vtilden} ^2
:= \sum_{\E \in \taun} \Norm{\nabla \vtilden}^2_{0,\E} 
+ \sum_{\F \in \Fcaln}
    \hF^{-1} \Norm{\PiLo \jump{\vtilden}}^2_{0,\F}.
\end{equation}
We consider and analyze the following method:
\begin{equation} \label{method:variable-degree}
\begin{cases}
    \text{find } \utilden \in \Vtildenzero \text{ such that}\\ 
    \antildeparent{\utilden}{\vtilden}
    =    ( f,\vtilden)_{0,\Omega}
    \qquad\qquad  \forall \vtilden \in \Vtildenzero.
    \end{cases}
\end{equation}
The following well-posedness and convergence result 
for method~\eqref{method:variable-degree} holds true.

\begin{proposition} \label{proposition:convergence-variable-order}
If~$\etaF$ is sufficiently large for all~$\F$ in~$\Fcaln$,
then method~\eqref{method:variable-degree} is well-posed.
Given~$u$ and~$\utilden$ be the solutions to~\eqref{weak:form} and~\eqref{method:variable-degree}, we have
\begin{equation}\label{strang:on-tilde-variable-degree}
\Normthreebarshp{u-\utilden}
\lesssim
\inf_{\vtilden \in \Vtildenzero} \Normthreebarshp{u-\vtilden} 
+ \sup_{\vtilden \in \Vtildenzero} 
    \frac{|\antildeparent{u}{\vtilden}
    -( f,\vtilden)_{0,\Omega}|}{\Normthreebarshp{\vtilden}}.
\end{equation}
If~$u$ belongs to $H^{s} (\Omega)$,
$s$ larger than~$3/2$,
and~$\pE$ smaller than or equal to~$s$
for all $\E$ in~$\taun$,
we have
\begin{equation} \label{convergence:variable-degree-method}
\Normthreebarshp{u-\utilden}^2
\lesssim  \sum_{\E \in \taun} 
    (\hpE)^2 \SemiNorm{u}_{\pEpo,\E}^2.
\end{equation}
\end{proposition}
\begin{proof}
The bilinear form in~\eqref{anptilde}
is coercive with respect to the mesh-dependent norm~\eqref{variable-degree:norm}.
Similarly to what we did in the proof
of Proposition~\ref{proposition:convergence-stabilized-method},
this fact is a consequence of arguments
that are standard in variable order
symmetric interior penalty DG methods;
we refer to~\cite[Lemma~4.12]{DiPietro-Ern:2012}.
On the one hand this implies the well-posedness of
method~\eqref{method:variable-degree}
for sufficiently large~$\etaF$
for all~$\F$ in~$\Fcaln$.

On the other hand, inequality~\eqref{strang:on-tilde-variable-degree}
follows from the second Strang's lemma;
see, e.g., \cite{Strang:1972,Brenner:2015}.
The hidden constants therein depends on the discrete coercivity constant, which in turns depends on the collection of
the stabilization parameters~$\etaF$
and the polynomial degree distribution~$\pbf$.

In order to derive estimate~\eqref{convergence:variable-degree-method},
we show an upper bound for the two terms on the right-hand side
of~\eqref{strang:on-tilde-variable-degree} separately.
As for the first one, we observe that
\[
\begin{split}
& \inf_{\vtilden \in \Vtildenzero} (\sum_{\E \in \taun} \Norm{\nabla (u-\vtilden)}^2_{0,\E} 
+ \sum_{\F \in \Fcaln} \hF^{-1} \Norm{\PiLo \jump{\vtilden}}^2_{0,\F}) \\
& \lesssim  \inf_{\vtilden \in \Vtildenzero} (\sum_{\E \in \taun} \Norm{\nabla (u-\vtilden)}^2_{0,\E} 
+ \sum_{\F \in \Fcaln} \hF^{-1} 
        \Norm{\jump{\vtilden}}^2_{0,\F}).
\end{split}
\]
The best error in CR spaces is smaller than
the error in $H^1$ conforming spaces;
in particular the jump terms vanish.
Optimal convergence follows from
polynomial approximation results~\cite{Schwab:1998}.

As for the second term in~\eqref{strang:on-tilde-variable-degree},
we observe that
$\average{\nabla u}$, $\jump{u}$,
and~$\PiLo\jump{\vtilden}$
are equal to~$\nabla u$, $0$, and $\jump{\vtilden}$,
respectively, on all edges.
These identities, an integration by parts,
problem~\eqref{strong-pois},
and the consistency of the bilinear form~\eqref{anptilde} yield
\footnotesize\begin{equation*}
\begin{split}
& \antildeparent{u}{\vtilden} -(f,\vtilden)_{0,\Omega} 
 = \sum_{\F \in \Fcaln}
    \left(\nbfF \cdot \nabla u,
    \jump{\vtilden} -\PitildeLo \jump{\vtilden} \right)_{0,\F}  
    - \sum_{\F \in \Fcaln} 
    \left( \PiLo (\normal{\F} \cdot \nabla u ),
    \PiLo \jump{\vtilden} )_{0,\F} \right) \\
& = \sum_{\F \in \Fcaln}
    \left(\nbfF \cdot \nabla u,
    \jump{\vtilden} - \PiltwoF{\pF-1} \jump{\vtilden} \right)_{0,\F}
  = \sum_{\F \in \Fcaln}
    \left(\nbfF \cdot \nabla u - \PiltwoF{\pF-1} (\normal{\F} \cdot \nabla u ),
    \jump{\vtilden}-\PiltwoF{0}\jump{\vtilden}\right)_{0,\F}.
\end{split}            
\end{equation*}\normalsize
The assertion now follows along the same line
as in the proof of Proposition~\ref{proposition:standard-convergence}.
\end{proof}

\begin{remark} \label{remark:alternative-even-fixed}
A low hanging fruit of the proofs
of Lemma~\ref{lemma:unisolvence-variable-order}
and Proposition~\ref{proposition:convergence-variable-order}
is that the construction of fixed even degree CR elements
as in Section~\ref{section:CR-even}
can be easily replaced by using spaces
of modal functions as above
and \emph{linear} (or also cubic, quintic, \dots
instead of order $\p-1$) nonconforming bubble functions;
the corresponding method should contain correction terms
acting on the edge projection onto a single scaled linear
(or cubic, quintic, \dots) Legendre polynomial.
Of course, one may also wish to construct
odd degree CR elements based on using linear nonconforming
bubbles.
In this case, the correction term,
which is needed to recover convergence of order~$\p$
and not for well-posedness reasons,
should be included in the method
as well for it to work.
The price to pay would be the presence of a stabilization term;
the upside would be the availability of
nested bases in a $\p$-refinement procedure.
\end{remark}

\begin{remark} \label{remark:hp-version}
In the proof of the error estimates~\eqref{convergence:variable-degree-method},
one may also keep track of the explicit dependence
on the polynomial degree distribution~$\pbf$.
In particular, one may derive
\begin{itemize}
\item algebraic convergence for uniformly increasing~$\p$ and singular solutions;
\item exponential convergence with respect to uniformly increasing~$\p$
for analytic solutions;
\item exponential convergence with respect to the cubic root
of the number of degrees of freedom
for solutions with corner singularities,
using sequences of meshes that are
geometrically refined towards the corners
of the domain, and distributions of polynomial degrees
that are linearly increasing
with the distance from the corners.
\end{itemize}
While we postpone the discussion of a practical investigation of this aspect to Section~\ref{subsection:nr-p-hp} below,
we refer to \cite{Schwab:1998, Babuska-Suri:1987}
for the appropriate $\h\p$-approximation and convergence results.
\end{remark}

\subsection{Numerical investigation of the \texorpdfstring{$\h\p$}{hp}-version of the CR method} \label{subsection:nr-p-hp}
Here, we assess the numerical performance of method~\eqref{method:variable-degree}
for the approximation of corner singularities using $\h\p$-refinements.

\paragraph*{Test case.}
We consider the $L$-shaped domain
$\Omega:= (-1,1)^2\backslash [ (0,1) \times (-1,0) ]$
and the exact solution
\begin{equation} \label{u3}
u_3(r,\theta) = r^{\frac23} \sin \Big( \frac23 \theta \Big).
\end{equation}
The source term and the boundary data
are computed accordingly.

\paragraph*{Meshes and polynomial degree distributions.}
We consider sequences of
shape-regular locally quasi-uniform simplicial meshes~$\{ \taunn \}$
obtained via geometric refinements
at the re-entrant corner of the L-shaped domain;
the first three elements in the sequence
are depicted in Figure~\ref{figure:hp-refinements},
where the geometric grading factor is~$1/2$.
In particular, the decomposition of each mesh~$\taunn$
consists of~$n+1$ layers defined recursively as follows:
$\Lcalz$ is the set of elements abutting~$(0,0)$;
\[
\Lcali := \{ \Eo \in \taunn \, | \, 
\overline{\Eo} \cap \overline{\Et} \ne \emptyset,
\, \Et \in \Lcalimo, \, \Eo \notin \cup_{k=0}^{i-1}\Lcalk \}
\qquad\qquad\qquad
i = 1, \dots ,n.
\]
In order to obtain~$\taunnpo$ from~$\taunn$,
exclusively the elements belonging to the
layer~$\mathcal L^0$ are refined.
As for the polynomial distribution~$\pbf$,
we set
\begin{equation} \label{polynomial-degree-distribution}
\text{$\pE = i+1$ 
\qquad for all~$\E$ \; in \; $\Lcali$,
\qquad $i = 0, \dots, n$.}
\end{equation}

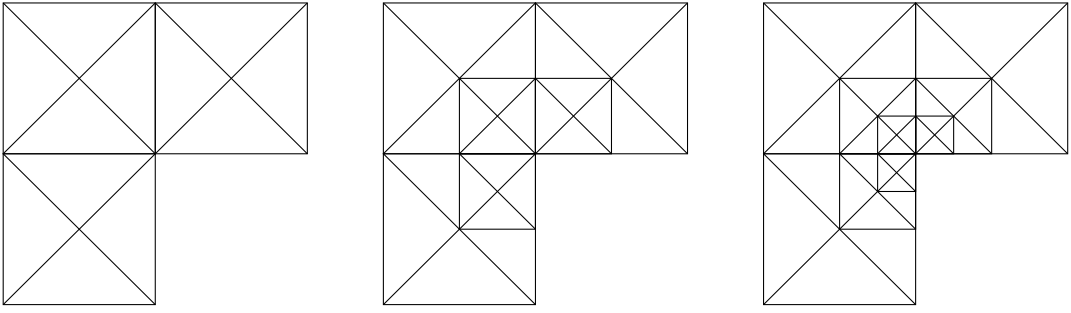
\begin{figure}[htpb]
\centering
\begin{tikzpicture}[scale=2]
\foreach[evaluate = \shift using 2.5*\i] \i in {0,1,2}{
\foreach[evaluate = \zeroup using 1-2^(2-\j)/4,
         evaluate = \twodown using 1+2^(2-\j)/4] \j in {0,...,\i}{
\draw [xshift=\shift cm] (1,1)  rectangle  (\twodown,\twodown);
\draw [xshift=\shift cm] (1,1)  rectangle  (\zeroup,\twodown);
\draw [xshift=\shift cm] (1,1)  rectangle  (\zeroup,\zeroup);
\draw [xshift=\shift cm] (\twodown,1) -- (1,\twodown) -- (\zeroup,1) -- (1,\zeroup);
\ifthenelse{\j = 0}{
\draw [xshift=\shift cm] (0,0) -- (2,2);
\draw [xshift=\shift cm] (0,2) -- (1,1);}}}
\end{tikzpicture}    
\caption{First three elements in a sequence
of geometrically graded meshes
(with grading parameter~$1/2$)
towards the re-entrant corner:
$\taunzn$ (left panel), $\taunon$ (center panel),
$\tauntw$ (right panel).}
\label{figure:hp-refinements}
\end{figure}

\paragraph*{Numerical result.}
In Figure~\ref{fig:variable-degree},
we display the results.

\begin{figure}[H]
\centering
\includegraphics[width=0.45\linewidth]{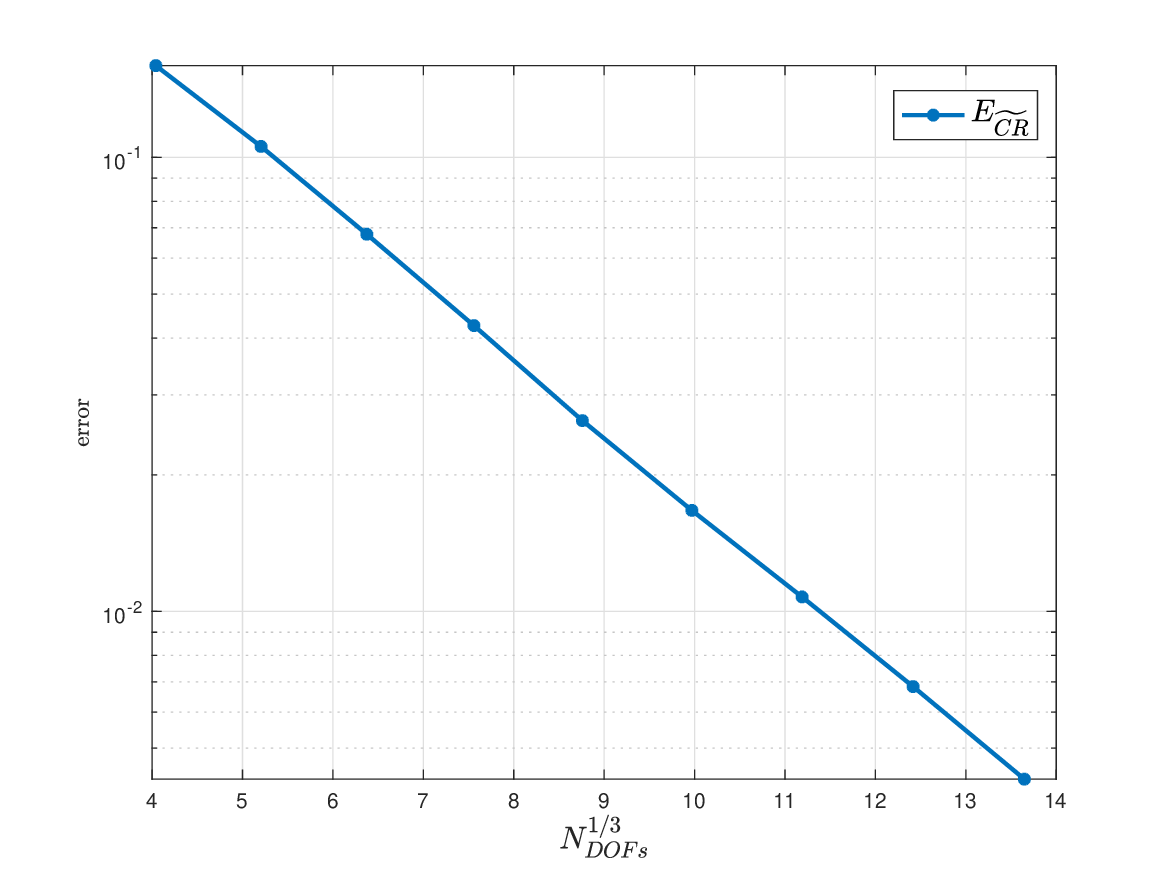} 
\caption{Exact solution~$u_3$ in~\eqref{u3};
we consider a sequence of geometrically graded meshes
(with grading parameter~$1/2$)
towards the re-entrant corner
as in Figure~\ref{figure:hp-refinements},
and linearly increasing polynomial degree distributions
as in~\eqref{polynomial-degree-distribution}.
We plot the CR error in~\eqref{error-measures}
versus the cubic root of the number of degrees of freedom.} \label{fig:variable-degree}
\end{figure}

Exponential convergence as discussed in Remark~\ref{remark:hp-version}
is observed.
A similar convergence is observed, and not reported,
for the $L^2$ error.

\section{New Crouzeix-Raviart elements
of even degree for the Stokes' equations} \label{section:inf-sup}
Crouzeix-Raviart elements were introduced~\cite{Crouzeix-Raviart:1973}
as an inf-sup stable, balanced,
and elementwise divergence free pair
for the discretization of the Stokes' problem.
Even though the main focus of this work
is on second order elliptic problems,
we are interested in designing
new even degree CR-type elements
and their practical stability.

For a given~$\fb$ in~$\Ltwod$, consider the Stokes' problem
\begin{equation} \label{Stokes-strong}
\begin{cases}
\text{find } (\ub,s) \text{ such that} \\
-\Deltab \ub  -\nabla s = \fb   & \text{in } \Omega \\
\dive \ub = 0                   & \text{in } \Omega \\
\ub = \zb                       & \text{on } \partial \Omega.
\end{cases}
\end{equation}
A weak formulation of~\eqref{Stokes-strong} reads as follows:
\begin{equation} \label{Stokes-weak}
\begin{cases}
\text{find } (\ub,s) \in \Honezd \times \Ltwoz =: \Vb \times Q \text{ such that} \\

(\nablab \ub, \nablab \vb){}_\zOmega + (\dive \vb, s)_\zOmega = (\fb , \vb)_\zOmega
                                   & \forall \vb \in \Vb \\
(\dive \ub, q)_\zOmega = 0         & \forall q \in Q. \\
\end{cases}
\end{equation}
The well-posedness and analysis of this problem is standard~\cite{Boffi-Brezzi-Fortin:2013}.

We consider vector-valued CR-type spaces~$\Vbftilden$ of order~$\p$,
see Section~\ref{subsection:new-CR-spaces},
and piecewise polynomials~$\Pbbpmotaun:=\Qn$ of order~$\p-1$
as approximation spaces to~$\Vb$ and~$Q$, respectively;
let~$\Bcaltilden$ be the vector version of the corresponding space as in~\eqref{method:with-BCs}.
We consider the following method,
which already incorporates
the imposition of the Dirichlet boundary condition
for the velocity
and the zero average condition for the pressure
as Lagrange multipliers:
\begin{equation} \label{stokes:CR-new}
\begin{cases}
\text{find } (\ubtilden,\bbftilden,\stilden, \lambdatilden) \in  \Vbftilden \times \Bcaltilden \times \Qtilden \times \Rbb \ \text{ such that} \\
\atilden(\ubtilden, \vbtilden){}_\zOmega
    + (\bbftilden,\vbtilden)_{0,\partial \Omega} 
    + (\stilden, \diveh \vbtilden)_\zOmega
    = (\fb , \vbtilden)_\zOmega
        & \forall \vbtilden \in \Vbftilden \\
(\ubtilden,\cbftilden)_{0,\partial \Omega}  = 0
        &  \forall \cbftilden \in \Bcaltilden \\
(\diveh \ubtilden, \qtilden)_\zOmega 
    + (\lambdatilden, \qtilden)_\zOmega 
    =  0         & \forall \qtilden \in \Qtilden \\
(\stilden,\mutilden)_\zOmega
    =  0         & \forall \mutilden \in \Rbb .
\end{cases}
\end{equation}
This method coincides with the standard CR method if~$\p$ is odd,
whereas for even~$\p$ it employs the novel CR-type space
as discussed in Section~\ref{subsection:new-CR-spaces}.

In matrix form, method~\eqref{stokes:CR-new} reads as follows (with obvious notation):
\[
\begin{pmatrix}
\Abfn     & \Cbfnstar & \Bbfnstar & \zerobf   \\
\Cbfn     & \zerobf   & \zerobf   & \zerobf   \\
\Bbfn     & \zerobf   & \zerobf   & \Dbfnstar \\
\zerobf   & \zerobf   & \Dbfn     & 0
\end{pmatrix}
\begin{pmatrix}
\ubfn \\
\bbfn \\
\sbfn \\
\lambda_n
\end{pmatrix}
=
\begin{pmatrix}
\fbfn \\
\zerobf \\
\zerobf \\
0
\end{pmatrix}
\]

\subsection{Numerical results: Stokes' equations} 
\label{nr:Stokes-convergence}
Here, we investigate the convergence of
method~\eqref{stokes:CR-new} on the numerical level.

\paragraph*{Error measures.}
In what follows,
$(\ubfn, s)$ and $(\ubtilden, \stilden)$
denote the solutions
to~\eqref{Stokes-weak} and~\eqref{stokes:CR-new}. 
We shall compute the following error measures:
\begin{equation} \label{error-measures:Stokes}
\EQtildeCRp :=  \frac{\Norm{s-\stilden}_{0,\Omega}}{\Norm{s}_{0,\Omega}}  ,
\qquad\qquad
\EVtildeCRp := \frac{\Normthreebars{\ub-\ubtilden}}{\SemiNorm{\ub}_{1,\Omega}} ,
\end{equation}
where~$\Normthreebars{\cdot}$ is
the vector-valued version of that
defined in~\eqref{discrete-norm}.

\paragraph*{Test case.}
We consider the polynomial degrees $\p=4$ and~$6$,
and the manufactured
solution~$\ub_4 = [(\ub_4)_1 ;(\ub_4)_2]$
and $s_4$ given by
\begin{equation} \label{u4}
\begin{split}
(\ub_4)_1(x,y)
& = -\cos \Big( \pi (x - \frac12) \Big)^2 \,
    \cos \Big( \pi (y - \frac12) \Big)^2 \,
    \sin \Big( \pi (y -\frac12) \Big)^2 ; \\
(\ub_4)_2(x,y)
& = \phantom{-} \cos \Big( \pi (x - \frac12) \Big)^2 \,
  \cos \Big( \pi (y - \frac12) \Big)^2 \,
  \sin \Big( \pi (x -\frac12) \Big)^2 ; \\
s_4(x,y)
& = \phantom{-} x+y-1.
\end{split}
\end{equation}
on the unit square $\Omega:=(0,1)^2$;
the loading term and boundary conditions are computed accordingly.
We consider the same sequence of meshes
employed in Section~\ref{subsection:conv-h-version}.

\paragraph*{Numerical results.}
In Figure~\ref{fig:Stokes},
we display the results.

\begin{figure}[H]
\centering
\includegraphics[width=0.37\linewidth]{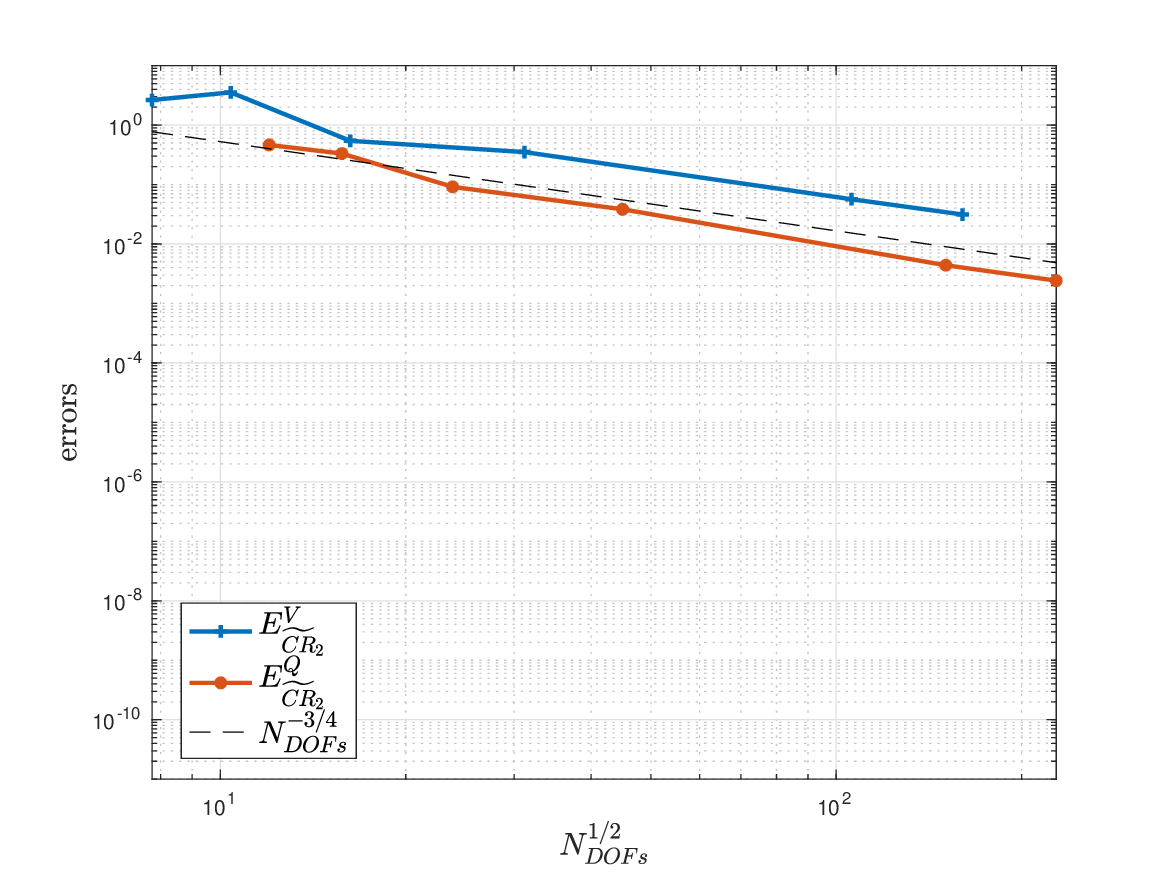}
\includegraphics[width=0.37\linewidth]{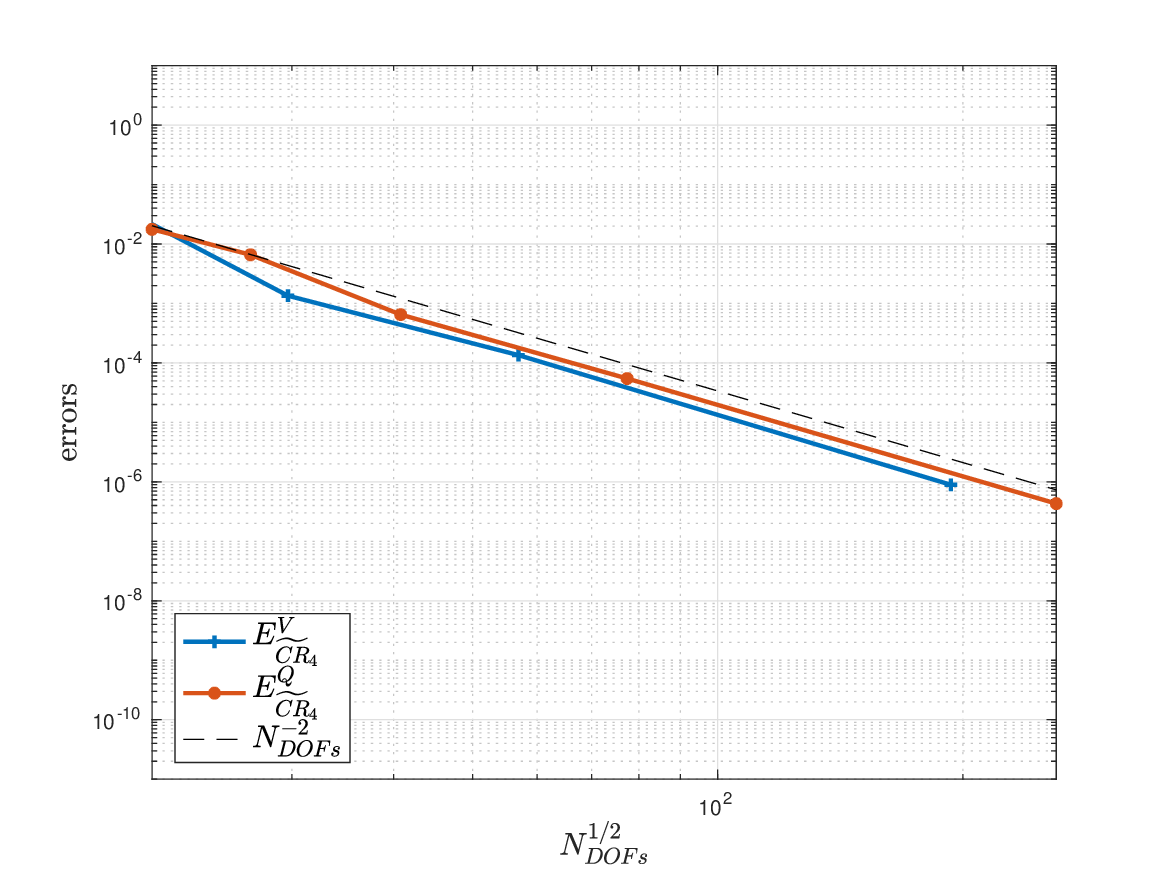}  \includegraphics[width=0.37\linewidth]{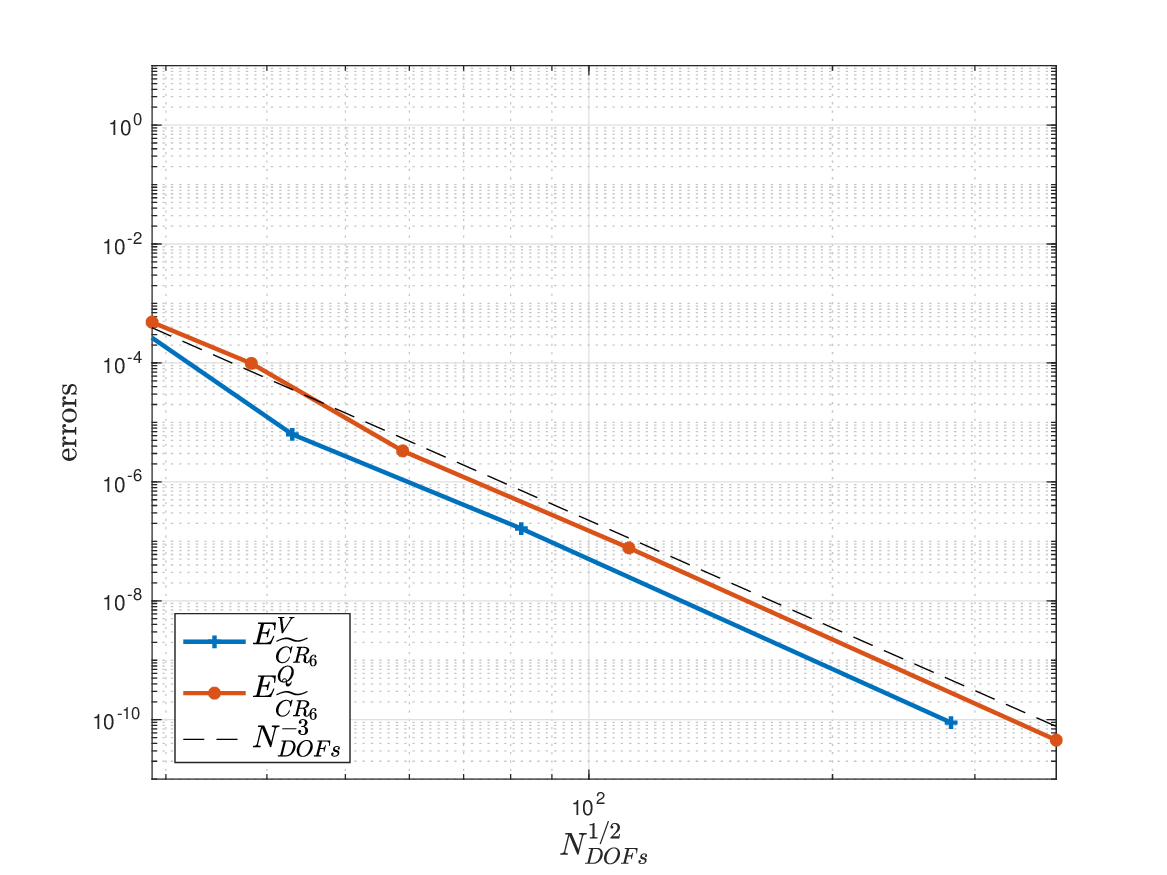}
\caption{Exact solution~$\ub_4$ and~$s_4$ in~\eqref{u4};
$\p=2$ (top-left panel);
$\p=4$ (top-right panel);
$\p=6$ (bottom panel);}
$\eta=20$.
We display the error measures in~\eqref{error-measures:Stokes}
under $\h$-refinements. \label{fig:Stokes}
\end{figure}

We observe convergence of order~$4$ and~$6$
for the methods of corresponding order;
as such, the method is inf-sup stable;
similar results are obtained for higher orders
but are not reported here.
On the other hand, we observe half an order suboptimal
convergence for the case~$\p=2$;
this and standard well-posedness results
for mixed problems suggest that the inf-sup constant
in this case is not robust with respect to the mesh
size and may be behaving like $\mathcal O(\h^\frac12)$.

We could have partially expected all these facts.
In fact, the proposed discrete pairs
lie in between the Scott-Vogelius ones \cite{scott-Vogelius:1985}
(which are stable for $\p$ larger than or equal to~$4$,
and meshes without singular vertices)
and that given by discontinuous velocities
and pressures~\cite{Baker-Jureidini-Karakashian:1990}:
the method for even degree larger than~$2$
must be inf-sup stable;
on the other hand, nothing could have been expected
a priori for the quadratic case,
since also the Scott-Vogelius element is not stable.

\section{Conclusions}
\label{section:conclusions}

We constructed new Crouzeix-Raviart (CR) elements
and corresponding methods
of even degree~$\p$ for a two dimensional Poisson problem:
such spaces are the span of modal functions
excluding those associated
with the vertices of the given mesh
and edge nonconforming bubble functions
of a suitable odd degree
thus mimicking what happens in the odd degree CR spaces;
the method required a DG-type stabilization
in order to deliver convergence rates of order~$\p$.

The reason for introducing such spaces was twofold:
\begin{itemize}
    \item for increasing polynomial order,
    we wanted to construct sequences of spaces
    spanned by functions with a ``hierarchical'' structure
    (or even nested bases);
    \item we wanted to establish the ground
    for a simple and optimal design of variable order CR elements.
\end{itemize}
The heart of the matter in the novel construction
is the observation that adding even degree
nonconforming bubble functions to the modal basis
yields a set of linearly dependent functions.
This is not the case while using odd degree
nonconforming bubbles,
the intrinsic reason being described
in displays \eqref{trick-odd-case}
and~\eqref{trick-even-case}.
The price to pay is the need of stabilizing
the method in order not to lose convergence.

Variable order CR elements were designed;
employing $\h\p$-refinements
yielded exponential convergence in terms
of the cubic root of the degrees of freedom
for solutions displaying a singular behaviour
at the corners of the physical domain.
Importantly,
\begin{itemize}
    \item for the construction of fixed even degree~$\p$
    CR elements, nonconforming bubbles
    of odd degree $1$ (or $3$, $5$, \ldots) can be employed;
    \item for the design of variable order
    CR elements,
    we use linear nonconforming bubbles so as to retain
    the linear consistency of the method;
    for this reason, we also had to stabilize the method
    on the edges of local odd degree
    larger than~$1$.
\end{itemize}
We further investigated whether the novel approach
could fit in an inf-sup stable
discretization of the Stokes' equations,
which preserves the elementwise
incompressibility constraint
on the discrete level.
This was in fact the case for even polynomial degrees
larger than~$2$; for the quadratic element,
we observed lack of inf-sup stability,
which resembles that for the quadratic
Scott-Vogelius element.

\paragraph*{Acknowledgments.}
The authors are grateful to the anonymous
referees whose comments allowed us
to improve a previous version of the manuscript.
A. Bressan has been partially funded by MUR (PRIN2022 research grant n. 2022A79M75).
L. Mascotto and M. Mosconi have been partially funded by MUR (PRIN2022 research grant n. 202292JW3F).
L. Mascotto has been partially funded by the
European Union (ERC, NEMESIS, project number 101115663);
views and opinions expressed are however those
of the authors only and do not necessarily reflect
those of the EU or the ERC Executive Agency.
The authors are also members of GNCS-INDAM.

{\footnotesize \bibliography{bibliogr}} \bibliographystyle{plain}

\end{document}